\documentclass[final,3p,times]{elsarticle}
\usepackage{graphicx}
\usepackage{caption}
\usepackage{subcaption}
\usepackage{color}
\usepackage{amssymb,amsmath,mathtools,amsthm}
\usepackage{stmaryrd}
\usepackage{bibentry}

\usepackage{tikz}
\usetikzlibrary{fit,positioning}
\usepackage[numbers]{natbib}

\usepackage[notcite,notref]{showkeys} 
\newcommand{\AJS}[1]{{\color{black}#1}}
\newcommand{\shlee}[1]{{\color{black}#1}}
\usepackage{mydef,boldfonts,xspace}


\newcommand{\wbu}{\widetilde{\bu}}
\newcommand{\bdf}[2]{{\textup{\textsf{BDF}}_{#2}\left({#1}\right)}}

\newcommand{\gu}{{\textup{GU}}}
\newcommand{\rot}{{\textup{ROT}}}

\nobibliography*

\begin{document}

\begin{frontmatter}

\title{Stability analysis of pressure correction schemes for the Navier-Stokes equations with traction boundary conditions}

\author[SL]{Sanghyun Lee}
\ead{shlee@ices.utexas.edu}
\address[SL]{Center for Subsurface Modeling, Institute for Computational Engineering and Sciences (ICES), University of Texas at Austin}

\author[AJS]{Abner J.~Salgado}
\ead{asalgad1@utk.edu}
\address[AJS]{Department of Mathematics, University of Tennessee, Knoxville, TN 37996 USA.}




\begin{abstract}
We present a stability analysis for two different rotational pressure correction schemes with \AJS{open and} traction boundary conditions. First, we provide a stability analysis for a rotational version of the grad-div stabilized scheme of [\bibentry{SLee2013}]. This scheme turns out to be unconditionally stable, provided the stabilization parameter is suitably chosen. We also establish a conditional stability result for the boundary correction scheme presented in [\bibentry{MR3253466}]. These results are shown by employing the equivalence between stabilized gauge Uzawa methods and rotational pressure correction schemes with traction boundary conditions. 
\end{abstract}


\end{frontmatter}

\section{Introduction}
\label{sec:Intro}

The purpose of this work is to provide a stability analysis of splitting schemes for the incompressible Navier Stokes equations (NSE)
\begin{equation}
\label{eq:NSE}
  \begin{dcases}
    \ue_t + \DIV(\ue \otimes \ue) - \DIV \left( \frac2\Reyn \vare(\ue) \right) + \GRAD \pe = \bef, & \text{in } \Omega \times(0,T], \\
    \DIV \ue = 0, & \text{in } \Omega\times(0,T].
  \end{dcases}
\end{equation}
Here $\Omega$ is a bounded domain in $\Real^d$ with $d \in \{2,3\}$, 
$\bef:\Omega \times (0,T]\rightarrow \mathbb R^d$ is a given smooth source term, \AJS{$\Reyn>0$ is the so-called Reynolds number, and $T>0$ is the final time. The operator $\vare(\ue)$ is assumed to take one of the following forms:
\[
  \vare(\ue) := \frac12 \begin{dcases}
                  \GRAD \ue, & \text{open boundary conditions}, \\
                  \GRAD \ue + \GRAD \ue^\intercal, & \text{traction boundary conditions}.
                \end{dcases}
\]} 
The dependent variables are the velocity of the fluid
$\ue:\Omega \times [0,T]\rightarrow \mathbb R^d$, and the pressure 
$\pe:\Omega \times [0,T] \rightarrow \mathbb R$.
The system is supplemented with the initial condition
\begin{equation}
\label{eq:IC}
  \ue_{|t=0} = \ue_0
\end{equation}
and suitable boundary conditions discussed below.

Time discretization algorithms for the solution of \eqref{eq:NSE}--\eqref{eq:IC} can be classified as either fully coupled or splitting techniques. Fully coupled methods are usually complicated by the fact that the incompressibility constraint induces a saddle point structure. It is not our intention here to provide an overview of the literature, so we only refer the reader to \cite{MR3235759}. On the other hand, splitting techniques are advocated because they are somewhat easier to implement, since they circumvent the saddle point structure. The reader is referred to the overview \cite{MR2250931} for details. However, when studying the properties of such splitting techniques it is usually assumed that the boundary conditions are \emph{no-slip}, that is
\[
  \ue_{|\Gamma} = \boldsymbol{0},
\]
where $\Gamma := \partial\Omega$ denotes the boundary of $\Omega$. These are not the only possible boundary conditions for \eqref{eq:NSE}, nor the only ones that are physically relevant. \AJS{Other possibilities, and the ones we are interested in here, are the so-called \emph{open and traction boundary conditions}
\begin{equation}
\label{eq:BC}
  \left( \frac2\Reyn\vare(\ue) - p \polI \right)_{|\Gamma} \SCAL \bn = \bg,
\end{equation}
where $\bn$ is the unit outward normal to $\Gamma$ and \shlee{$\bg:\Gamma \times (0,T]\rightarrow \mathbb R^d$} is a given function. Open boundary conditions, that is the case where $\vare(\ue) = \tfrac12 \GRAD \ue$ and $\bg \equiv \boldsymbol{0}$, arise when dealing with outflow or artificial boundaries \cite{MR1423081,MR3327983}. On the other hand, for free surface flows \cite{MR1826211,FLD:FLD4071}, one usually prescribes traction boundary conditions, \ie $\vare(\ue) = \tfrac12 (\GRAD \ue + \GRAD \ue^\intercal)$ and $\bg$ is related to the surface tension}.

The development and analysis of splitting schemes for traction boundary conditions is rather scarce. To our knowledge, the first reference that provides a rigorous analysis of such methods is \cite{MR2177143}, where the authors show that these suffer from a drastic accuracy reduction. Several attempts to remedy this shortcoming can be found in the literature. For instance, \cite{MR3327983} proposes a modification of the boundary condition \eqref{eq:BC} aimed at remedying issues related to backflows and numerical instabilities that appear at open boundary conditions for high Reynolds numbers. In the course of their discussion, the authors were able to show unconditional stability of a formally first order class of schemes with modified boundary conditions. Other modifications of the boundary conditions \eqref{eq:BC} also exist in the literature \cite{MR3253466,MR2988624} which seem to work in practice, but lack \shlee{of} any analytical justification. \shlee{Some other works modify the splitting scheme itself and this allows them to recover optimal experimental orders of convergence, e.g \cite{MR2957735,MR2855967,Angot12,SLee2013}.} While many other approaches can be found in the literature, in this work we focus on two of them:
\begin{enumerate}[1.]
  \item Using a so-called grad-div stabilization \cite{SLee2013} presents a first order pressure correction method in standard form and shows its unconditional stability. Here we extend this scheme by considering its rotational form and show that its first and second order versions are unconditionally stable.

  \item Reference \cite{MR3253466} presents a variant of the so-called (using the nomenclature of \cite{MR2250931}) rotational pressure correction method that seems to deliver optimal orders of convergence. However, no analysis of this method is provided. It is our purpose here to partially bridge this gap.
\end{enumerate}

Let us outline the plan that we will follow to achieve these goals. 
In Section \ref{sec:equs}, we recall the recently shown equivalence \cite{Pyo} between rotational pressure correction schemes and stabilized gauge Uzawa methods. With this at hand we recast the grad-div stabilized schemes of \cite{SLee2013} in gauge Uzawa form and show unconditional stability for the first and second order variants in Section \ref{sec:Sanghyun}. We then write the scheme of \cite{MR3253466} in gauge Uzawa form and show that it is stable provided the time step and mesh size satisfy a suitable relation in Section \ref{sec:Baensch}. 
In Section \ref{sec:numerics} we provide numerical experiments to illustrate the theory and performance of our methods.

\subsection{Notation and preliminaries}
\label{sub:notation}

For $D \in \{ \Omega,\Gamma\}$, we denote by $\scl\cdot,\cdot\scr_D$ the $L^2(D)$ inner product, if $D=\Omega$ we will often omit it. The norm of $L^2(\Omega)$ is denoted by $\|\cdot\|$. To shorten the exposition, given two scalar functions $\phi$ and $\varphi$, we set
\begin{equation}
\label{eq:trnorm}
  \sbl \phi,\varphi \sbr := \scl \phi,\varphi \scr + \scl \GRAD \phi, \GRAD \varphi \scr \text{ and }  
  \trnorm{\varphi}^2 := \sbl \varphi, \varphi \sbr.
\end{equation}
We denote nonessential constants by $C$ and their value might change at each occurrence. To avoid irrelevant technicalities we omit the convective term $\DIV(\ue \otimes \ue)$.
\shlee{In addition, we recall the elementary identity}
\[
  \AJS{2p(p-q) = p^2 - q^2 + (p-q)^2}.
\]
\AJS{To unify the discussion, we define the bilinear form
\begin{equation}
\label{eq:defofcalA}
  \calA \scl \bv,\bw\scr := \frac1\Reyn  \begin{dcases}
                    \scl \GRAD \bv, \GRAD \bw\scr, & \text{open boundary conditions}, \\
                    \scl \vare(\bv), \vare(\bw) \scr, & \text{traction boundary conditions},
                   \end{dcases}
\end{equation}
and set
\[
  \calA(\bv) := \calA\scl\bv,\bv\scr^{1/2}.
\]}
We denote by $\kappa$ the best constant in the inequality
\begin{equation}
\label{eq:Korn}
  \kappa \AJS{d} \left( \| \bv \|^2 + \| \GRAD \bv \|^2 \right) \leq \| \bv \|^2 + \AJS{\Reyn \calA(\bv)^2}, \quad \forall \bv \in \Hund,
\end{equation}
\AJS{which is either trivial (open boundary conditions) or is a consequence of Korn's inequality (traction boundary conditions)}.
By $\|\cdot\|_{1/2,\Gamma}$ we denote the norm of $\bH^{1/2}({\Gamma})$, the space of traces of $\Hund$ on $\Gamma$. We recall that
\begin{equation}
\label{eq:trace}
  \| \bv\|_{1/2,\Gamma} \leq C \left( \| \bv \|^2 + \| \GRAD \bv \|^2 \right)^{1/2}.
\end{equation}

To handle the space discretization, 
we assume we have at hand finite dimensional spaces $\bX_h \subset \Hund$ and $Q_h \subset \Hun$ which we assume LBB stable 
for $h>0$.
We set $M_h = Q_h \cap \Hunz$. These spaces can be easily realized with finite elements and, in this case, $h$ denotes the mesh size. We assume that the following inverse inequality holds
\begin{equation}
\label{eq:invineq}
  \| \partial_\bn w_h \|_{1/2,\Gamma} \leq C h^{-1} \| \GRAD w_h \|, \quad \forall w_h \in M_h .
\end{equation}

The time discretization is carried out by choosing $\calK\in \polN$, the number of time steps, and setting the time step to be $\dt = T/\calK$. We set $t_k = k\dt$ and for a time dependent function we denote $\phi^k = \phi(t_k)$ and $\phi^\dt = \{ \phi^k \}_{k=0}^\calK$. Over these sequences we define the operators
\begin{equation}
\label{eq:defofdelta}
  \frakd \phi^{k+1} := \phi^{k+1} - \phi^k,
\end{equation}
\begin{equation}
\label{eq:defofbdf}
  \bdf{\phi^{k+1}}{m} := \begin{dcases}
                       \frac1{\dt} \frakd \phi^{k+1} & m = 1, \\
                       \frac1{2\dt} \left( 3 \phi^{k+1} - 4\phi^k + \phi^{k-1} \right) & m = 2,
                     \end{dcases}
\end{equation}
and
\begin{equation}
\label{eq:defofsharp}
  \phi^{\sharp,k} := \begin{dcases}
                       \frakd \phi^k & m = 1, \\
                       \frac43 \frakd \phi^k - \frac13 \frakd \phi^{k-1} & m = 2.
                \end{dcases}
\end{equation}
for $m = 1,2$.
{If $E$ is a normed space with norm $\|\cdot \|_E$ and} $\phi^{\dt} \subset E$, we define the following discrete (in time) norms:
\begin{equation}
\| \phi^{\tau} \|_{\ell^2(E)} := \left( \dt \sum^\calK_{k=1} \|\phi^k\|^2_E \right)^{1/2}, \qquad 
\| \phi^{\tau}\|_{\ell^{\infty}(E)} := \max_{0 \leq k \leq \calK} ( \| \phi^k \|_E ).
\end{equation}

\shlee{Next, we introduce couple of lemmas which are useful for our analyses.} \AJS{In the analysis of first order schemes we will use the following variant of the well-known discrete Gr\"onwall inequality \cite[Lemma 5.1]{MR1043610}:

\begin{lem}[Discrete Gr\"onwall]
\label{lem:gronwall}
Let $B\geq 0$ and $a^\dt,b^\dt,c^\dt,\gamma^\dt \subset \Real$ be sequences of nonnegative numbers such that, for $n\geq 0$, verify
\[
  a^n + \dt \sum_{k=1}^n b^k \leq \dt \sum_{k=0}^{n-1} \gamma^k a^k + \dt \sum_{k=0}^{n-1} c^k + B.
\]
If $\dt \gamma^k<1$ for all $k \geq 0$, then we have
\[
  a^n + \dt \sum_{k=1}^n b^k \leq \exp\left( \dt \sum_{k=0}^n \sigma^k \gamma^k \right) \left( \dt \sum_{k=0}^{n-1} c^k + B \right),
\]
where $\sigma^k = (1-\dt \gamma^k)^{-1}$.
\end{lem}

\shlee{On the other hand, we will employ the following three-term recursion inequalities of \cite[Propositions 5.1 and 5.2]{GuermondSalgadoErrorAnalysis} for the analysis of second order schemes. }

\begin{lem}[Three-term recursion]
\label{lem:threeterm}
Let $A,B,C \in \Real$ satisfy $A>0$, $C \geq 0$ \shlee{with} $A+B+C \leq 0$ and assume that the quadratic equation $Ar^2+Br+C=0$ has two nonzero real roots $r_1$ and $r_2$. \shlee{In addition,} let $a^\dt$ solve the inequality
\[
  Aa^{k+1} + B a^k + Ca^{k-1} \leq g^{k+1}, \qquad k \geq 1.
\]
Then there are constants $c_1$ and $c_2$ that depend only on $a^0$ and $a^1$ such that, for every $n \geq 2$, we have
\[
  a^n \leq c_1 r_1^n + c_2 r_2^n + \frac1A \sum_{k=2}^n r_1^{n -k} \sum_{s=2}^k r_2^{k-s} g^s.
\]
\end{lem}
}

\section{Stabilized gauge Uzawa equals rotational pressure correction}
\label{sec:equs}

\AJS{We begin by recalling the result of \cite[\S3.6]{MR2250931}: when supplemented with no-slip boundary conditions, the rotational pressure correction method of Timmermans \etal \cite{TMV96} is equivalent to the scheme of Kim and Moin \cite{KIM1985308}. More recently, Pyo \cite{Pyo} showed that, in the same setting,} the gauge Uzawa method and the rotational pressure correction method are, up to a change of variables, equivalent. Let us recall this result in this section.

\subsection{Equivalence}
\label{sub:equiv}
Since we are concerned with time discretization schemes, let us operate in a semi-discrete setting to show the equivalence. To simplify things even further, we will only consider first order schemes ($m=1$). Little or no modification is necessary if the equivalence wants to be shown for $m=2$.

\subsubsection{The stabilized gauge Uzawa method}
This method was introduced in \cite{MR2177795} and computes sequences $\wbu_\gu^\dt$, $\bu_\gu^\dt$, $p_\gu^\dt$, $\psi^\dt$ and $q^\dt$ as follows:

\begin{enumerate}[$\bullet$]
  \item \textbf{Initialization:} Set
  \[
    \wbu_\gu^0 = \bu_\gu^0 = \ue^0, \quad p_\gu^0 = \pe^0, \quad \psi^0 = q^0 = 0.
  \]
\end{enumerate}

For $k=0,\ldots,\calK-1$, we compute:

\begin{enumerate}[$\bullet$]
  \item \textbf{Velocity update:} Find $\bu_\gu^{k+1}$ that solves
  \begin{equation}
  \label{eq:velgu}
    \frac{ \bu_\gu^{k+1} - \wbu_\gu^k}\dt - \frac2\Reyn \DIV( \vare(\bu_\gu^{k+1}) ) + \GRAD p_\gu^k = \bef^{k+1}, \qquad 
    \AJS{\bu_{\gu|\Gamma}^{k+1}=\boldsymbol{0}}.
  \end{equation}
  
  \item \textbf{Projection:} Find $\psi^{k+1}$ and $\wbu^{k+1}_{{\gu}}$ that satisfy
  \begin{equation}
  \label{eq:projgu}
    \frac{ \wbu_\gu^{k+1} - \bu_\gu^{k+1} }\dt + \GRAD \frakd \psi^{k+1} = \boldsymbol{0}, \qquad \DIV \wbu_\gu^{k+1} = 0, \qquad \AJS{\wbu_{\gu|\Gamma}^{k+1} \cdot \bn =0}.
  \end{equation}
  
  \item \textbf{Pressure update:} Find $q^{k+1}$ and $p^{k+1}_{{\gu}}$ that satisfy
  \begin{equation}
  \label{eq:presgu}
    \frakd q^{k+1} = -\DIV \bu_\gu^{k+1}, \qquad p_\gu^{k+1} = \psi^{k+1} + \frac1\Reyn q^{k+1}.
  \end{equation}
\end{enumerate}

\subsubsection{Rotational pressure correction method}
This method was introduced in \cite{TMV96}. It computes sequences $\wbu_\rot^\dt$, $\bu_\rot^\dt$, $p_\rot^\dt$ and $\phi^\dt$ as follows:
\begin{enumerate}[$\bullet$]
  \item \textbf{Initialization:} Set
  \[
    \wbu_\rot^0 = \bu_\rot^0 = \ue^0, \quad p_\rot^0 = \pe^0, \quad \phi^0 = 0.
  \]
For $k=0,\ldots,\calK-1$, we compute:
  
  \item \textbf{Velocity update:} Find $\bu_\rot^{k+1}$ that solves
  \begin{equation}
  \label{eq:velrot}
    \frac{ \bu_\rot^{k+1} - \wbu_\rot^k}\dt - \frac2\Reyn \DIV( \vare(\bu_\rot^{k+1}) ) + \GRAD p_\rot^k = \bef^{k+1}, \qquad 
    \AJS{\bu_{\rot|\Gamma}^{k+1}=\boldsymbol{0}}.
  \end{equation}
  
  \item \textbf{Projection:} Find $\phi^{k+1}$ and $\wbu^{k+1}_{{\rot}}$ that satisfy
  \begin{equation}
  \label{eq:projrot}
    \frac{ \wbu_\rot^{k+1} - \bu_\rot^{k+1} }\dt + \GRAD \phi^{k+1} = \boldsymbol{0}, \qquad \DIV \wbu_\rot^{k+1} = 0, \qquad \AJS{\wbu_{\rot|\Gamma}^{k+1} \cdot \bn =0}.
  \end{equation}
  
  \item \textbf{Pressure update:} The pressure $p^{k+1}_{{\rot}}$ is defined by
  \begin{equation}
  \label{eq:presrot}
    \frakd p_\rot^{k+1} = \phi^{k+1} - \frac1\Reyn \DIV \bu_\rot^{k+1}.
  \end{equation}

\end{enumerate}

The equivalence of these two methods is the content of this simple, yet illuminating, result.

\begin{prop}[Equivalence]
\label{prop:equiv}
The sequences produced by algorithms \eqref{eq:velgu}--\eqref{eq:presgu} and \eqref{eq:velrot}--\eqref{eq:presrot} verify $\wbu_\gu^\dt = \wbu_\rot^\dt$, $\bu_\gu^\dt = \bu_\rot^\dt$ and $p_\gu^\dt = p_\rot^\dt$.
\end{prop}
\begin{proof}
Evidently, the initialization steps coincide. The velocity update steps \eqref{eq:velgu} and \eqref{eq:velrot} also coincide. Defining $\phi^{k+1} = \frakd \psi^{k+1}$ makes the projection steps \eqref{eq:projgu} and \eqref{eq:projrot} identical. Finally, applying the operator $\frakd$ to the second equation of \eqref{eq:presgu}, using the first equation and the definition of $\phi^{k+1}$ yields \eqref{eq:presrot}. This allows us to conclude.
\end{proof}

In light of this equivalence, in what follows we will drop the subscripts $\gu$ and $\rot$.

\subsection{Stabilized gauge Uzawa with elimination of the solenoidal velocity}
\label{sub:impl}
The projection step \eqref{eq:projgu} entails finding a solenoidal function $\wbu^\dt$, which can be rather cumbersome to approximate using finite elements. For this reason, in practice, this variable is usually eliminated from the scheme. This can be achieved as follows: Add to the velocity step \eqref{eq:velgu} the first equation in \eqref{eq:projgu} at time $t=t_k$, this eliminates $\wbu^k$ from the velocity step. Then, taking the divergence of the first equation in \eqref{eq:projgu} and using the second one 
we substitute $\wbu^{k+1}$ in the projection step. This yields the following algorithm:
\begin{enumerate}[$\bullet$]
  \item \textbf{Initialization:} Set
  \[
    \bu^0 = \ue^0, \quad p^0 = \pe^0, \quad \psi^0 = q^0 = 0.
  \]
  For $k=0,\ldots,\calK-1$, we compute:
  
  \item \textbf{Velocity update:} Find $\bu^{k+1}$ that solves
  \begin{equation}
  \label{eq:velguimp}
    \frac{ \bu^{k+1} - \bu^k}\dt - \frac2\Reyn \DIV( \vare(\bu^{k+1}) ) + \GRAD( p^k + \frakd \psi^k ) = \bef^{k+1}, \qquad 
    \AJS{\bu_{|\Gamma}^{k+1} = \boldsymbol{0}}.
  \end{equation}
  
  \item \textbf{Projection:} Find $\psi^{k+1}$ that satisfies
  \begin{equation}
  \label{eq:projguimp}
    \LAP \frakd \psi^{k+1} = \frac1\dt \DIV \bu^{k+1}, \qquad \AJS{\partial_\bn \frakd \psi^{k+1}_{|\Gamma} = 0}.
  \end{equation}
  
  \item \textbf{Pressure update:} Find $q^{k+1}$ and $p^{k+1}$ that satisfy
  \begin{equation}
  \label{eq:presguimp}
    \frakd q^{k+1} = -\DIV \bu^{k+1}, \qquad p^{k+1} = \psi^{k+1} + \frac1\Reyn q^{k+1}.
  \end{equation}
\end{enumerate}

\AJS{As the reader can easily realize, the choice of no-slip boundary conditions carried little or no relevance in the discussions above. Therefore, a similar reasoning could be provided for other types of boundary conditions}. The analysis of variants of scheme \eqref{eq:velguimp}--\eqref{eq:presguimp} that are suitable for \AJS{open and} traction boundary conditions is the main content of this work.

\section{A grad-div stabilized scheme}
\label{sec:Sanghyun}

Let us now turn to an extension of the scheme discussed in \cite{SLee2013}. In this work the authors modified the pressure correction scheme in standard form by introducing a grad-div stabilization and consistent modifications of the boundary condition in the velocity and pressure. This allowed them to obtain stability and optimal error estimates for a first order scheme. It is expected then, that a rotational version of this scheme allows us to obtain higher order schemes. The main purpose of this section is to present the first steps in this direction, namely we present the scheme and show its unconditional stability.

We compute sequences $\bu_h^\dt \subset \bX_h$, $\psi_h^\dt,q_h^\dt,p_h^\dt \subset Q_h$, where $\bu_h^\dt$ and $p_h^\dt$ approximate the velocity and pressure, respectively. The scheme reads:
\begin{enumerate}[$\bullet$]
  \item \textbf{Initialization:} Set, for $k=0, \ldots, m-1$,
  \begin{equation}
  \label{eq:init}
    \bu_h^k = \Pi_{\bX_h} \ue^k, \quad p_h^k = \Pi_{Q_h} \pe^k, \quad q_h^k = \psi_h^k = 0,
  \end{equation}
  where $\Pi_{\bX_h}$ and $\Pi_{Q_h}$ are the $L^2$-projection operators onto the respective spaces. After this step, for $k = 0,\ldots, \calK-1$, compute:

  \item \textbf{Velocity update:} Find $\bu_h^{k+1} \in \bX_h$ that solves, for all $\bv_h \in \bX_h$,
  \begin{equation}
  \label{eq:velsl}
    \scl \bdf{\bu_h^{k+1}}{m}, \bv_h \scr + \AJS{\calA \scl \bu_h^{k+1}, \bv_h \scr} - \scl p_h^k + \psi_h^{\sharp,k}, \DIV \bv_h \scr +
    \alpha \scl \DIV \bdf{ \bu_h^{k+1} }{m}, \DIV \bv_h \scr =
    \scl \bef^{k+1}, \bv_h \scr + \scl \bg^{k+1}, \bv_h \scr_\Gamma,
  \end{equation}
  where $\bdf{\cdot}{m}$, \AJS{$\calA\scl\cdot,\cdot\scr$} and $(\cdot)^{\sharp,k}$ where defined in \eqref{eq:defofbdf}, \AJS{\eqref{eq:defofcalA}} and \eqref{eq:defofsharp}, respectively. The parameter $\alpha \geq 1$ is user defined.
  
  \item \textbf{Projection:} Find $\psi_h^{k+1} \in Q_h$ that, for every $z_h \in Q_h$, satisfies
  \begin{equation}
  \label{eq:projsl}
    \sbl \frakd \psi_h^{k+1}, z_h \sbr = -\frac\beta\dt \scl \DIV \bu_h^{k+1}, z_h \scr,
  \end{equation}
  where $\beta = 1 + \tfrac12( m -1 )$.

  \item \textbf{Divergence correction:} Find $q_h^{k+1} \in Q_h$ that solves
  \begin{equation}
  \label{eq:divsl}
    \scl \frakd q_h^{k+1}, r_h \scr = -\scl \DIV \bu_h^{k+1}, r_h \scr \quad \forall r_h \in Q_h.
  \end{equation}
  
  \item \textbf{Pressure update:} The new approximation of the pressure $p_h^{k+1} \in Q_h$ is obtained by setting
  \begin{equation}
  \label{eq:pressl}
    p_h^{k+1} = \psi_h^{k+1} + \frac1\Reyn q_h^{k+1}.
  \end{equation}

\end{enumerate}

\subsection{Stability of the first order scheme}
\label{sub:fosl}
Let us now show the stability of scheme \eqref{eq:velsl}--\eqref{eq:pressl} for $m=1$. This is mainly done to clarify the steps necessary to obtain the result. Our main objective is to show the stability for the second order variant. This will be carried out in \S\ref{sub:sosl}. To avoid irrelevant technicalities, we assume that $\bef \equiv \boldsymbol{0}$ and $\bg \equiv \boldsymbol{0}$.

\begin{thm}[Stability for $m=1$]
\label{thm:stabfosl}
Assume that $\alpha > \max\{1, 2/\Reyn\}$. If $m = 1$ {and $\dt<\tfrac12$} then the scheme \eqref{eq:velsl}--\eqref{eq:pressl} is stable, in the sense that there is a constant $C$, independent of the solution and discretization parameters such that,
\[
  \| \bu_h^\dt \|_{\ell^\infty(\bL^2)} + \AJS{\| \calA(\bu_h^\dt) \|_{\ell^2(\Real)}} \leq C.
\]
The constant, however, might depend on $\Reyn$.
\end{thm}
\begin{proof}
First, we set $\bv_h = 2\dt \bu_h^{k+1}$ in \eqref{eq:velsl}, then multiply \eqref{eq:pressl} by $2\dt \Pi_{Q_h}\DIV \bu_h^{k+1}$ at time $t=t_k$ and finally set $r_h = 2\tfrac\dt\Reyn q_h^k$ in \eqref{eq:divsl}. Adding the ensuing identities yields
\begin{multline}
\label{eq:fo1sl}
  \frakd \left(\| \bu_h^{k+1}\|^2 +  \alpha \| \DIV \bu_h^{k+1} \|^2 + \frac\dt\Reyn \|q_h^{k+1}\|^2  \right) + 
  \| \frakd \bu_h^{k+1} \|^2 + \alpha \|\frakd \DIV \bu_h^{k+1} \|^2 \\
  + \AJS{2\dt \calA(\bu_h^{k+1})^2 }
  - 2 \dt \scl \psi_h^k + \frakd \psi_h^k, \DIV \bu_h^{k+1} \scr 
  = \frac\dt\Reyn \| \frakd q_h^{k+1} \|^2.
\end{multline}
Next, set $z_h = 2\dt^2( \psi_h^k + \frakd \psi_h^k ) \AJS{=\psi_h^{k+1}-\frakd^2 \psi_h^{k+1}}$ in \eqref{eq:projsl} to obtain
\begin{equation}
\label{eq:fo2sl}
  \dt^2 \left( \frakd \trnorm{\psi_h^{k+1}}^2 + \trnorm{\frakd \psi_h^k}^2 - \trnorm{ \frakd^2 \psi_h^{k+1} }^2 \right) = 
  -2\dt \scl \psi_h^k + \frakd \psi_h^k, \DIV \bu_h^{k+1} \scr,
\end{equation}
where $\trnorm{\cdot}$ was defined in \eqref{eq:trnorm}.
Now we add \eqref{eq:fo1sl} and \eqref{eq:fo2sl} to get
\begin{multline}
\label{eq:fo3sl}
  \frakd \left(\| \bu_h^{k+1}\|^2 +  \alpha \| \DIV \bu_h^{k+1} \|^2 + \frac\dt\Reyn \|q_h^{k+1}\|^2  + \dt^2 \trnorm{ \psi_h^{k+1} }^2 \right) + 
  \| \frakd \bu_h^{k+1} \|^2 + \alpha \|\frakd \DIV \bu_h^{k+1} \|^2 + \dt^2 \trnorm{\frakd \psi_h^k}^2 \\
  + \AJS{2\dt \calA(\bu_h^{k+1})^2 }
  = \frac\dt\Reyn \| \frakd q_h^{k+1} \|^2
  + \dt^2 \trnorm{ \frakd^2 \psi_h^{k+1} }^2.
\end{multline}
It remains to control the terms on the right hand side of \eqref{eq:fo3sl}. This can be achieved {as follows}. First, we apply $\frakd$ to \eqref{eq:projsl} and set $z_h = \frakd^2 \psi_h^{k+1}$ to derive
\begin{equation}
\label{eq:fo4sl}
  \dt^2 \trnorm{\frakd^2 \psi_h^{k+1} }^2 \leq \| \frakd \DIV \bu_h^{k+1} \|^2.
\end{equation}
Second, we set $r_h = \frakd q_h^{k+1}$ in \eqref{eq:divsl} to conclude that
\begin{equation}
\label{eq:fo5sl}
  \| \frakd q_h^{k+1} \|^2 \leq \| \DIV \bu_h^{k+1} \|^2.
\end{equation}
Combining \eqref{eq:fo3sl}--\eqref{eq:fo5sl} yields
\begin{multline}
\label{eq:almostfinfo}
  \| \bu_h^{k+1}\|^2 +  \alpha\left(1- \frac\dt2 \right) \| \DIV \bu_h^{k+1} \|^2 + \frac\dt\Reyn \|q_h^{k+1}\|^2  + \dt^2 \trnorm{ \psi_h^{k+1} }^2  + 
  \| \frakd \bu_h^{k+1} \|^2 + \AJS{2\dt \calA(\bu_h^{k+1})^2 } \\
  + \dt^2 \trnorm{\frakd \psi_h^k}^2
  \leq \| \bu_h^k\|^2 +  \alpha \| \DIV \bu_h^k \|^2 + \frac\dt\Reyn \|q_h^k\|^2  + \dt^2 \trnorm{ \psi_h^k }^2, 
\end{multline}
where we used that $\alpha > \max\{1, 2/\Reyn\}$.
\AJS{We now rewrite the previous inequality in a form that makes it suitable to apply Lemma~\ref{lem:gronwall}. Define
\begin{align*}
  a^k &= \| \bu_h^k\|^2 +  \alpha \| \DIV \bu_h^k \|^2 + \frac\dt\Reyn \|q_h^k\|^2  + \dt^2 \trnorm{ \psi_h^k }^2, \\
  b^k &= 2\calA(\bu_h^k)^2,
\end{align*}
and notice that \eqref{eq:almostfinfo} can then be rewritten as
\[
  \left( 1 - \frac\dt2 \right) a^{k+1} + \dt b^{k+1} \leq \left( 1 - \frac\dt2 \right) a^k + \frac\dt2 a^k.
\]
\shlee{Finally, we add this inequality over $k=0, \cdots, \calK-1$} and multiply the result by $\left( 1 - \frac\dt2 \right)^{-1}$ to obtain
\[
  a^\calK + \dt \sum_{k=1}^\calK b^k \leq a^0 + \dt \sum_{k=0}^{\calK-1} \frac1{2\left( 1 - \frac\dt2 \right)}a^k.
\]
\shlee{Since $\dt < \tfrac12$, we can apply the framework of Gr\"onwall's inequality in the form given by Lemma~\ref{lem:gronwall} by setting $\gamma^k=\gamma = \frac1{2\left( 1 - \frac\dt2 \right)}$, $c^\dt = 0$ and $B = a^0$.}
 Therefore, we have
\[
  a^\calK + \dt \sum_{k=1}^\calK b^k \leq a^0 \exp\left( T \sigma \gamma\right),
\]
where $\sigma = (1-\dt \gamma)^{-1}$ and 
\[
  \sigma \gamma= \frac\gamma{1-\dt\gamma} = \frac1{2-2\dt} < 1.
\]
This immediately allows us to conclude.
}
\end{proof}

\begin{rem}[Alternative pressure update step]
\label{rem:newpresup}
Following the course of the proof of Theorem~\ref{thm:stabfosl} the reader can easily verify that we can replace \eqref{eq:pressl} by
\[
  p_h^{k+1} = \psi_h^{k+1} + \frac\kappa\Reyn q_h^{k+1}
\]
and retain stability. Indeed, to obtain an alternative to \eqref{eq:fo1sl} one now needs to set $r_h = 2\tfrac{\kappa\dt}\Reyn q_h^k$ in \eqref{eq:divsl}. The first term on the right hand side of \eqref{eq:fo3sl} will now be multiplied by $\kappa$, we can then invoke \AJS{inequality \eqref{eq:Korn}} to conclude.
\end{rem}

\subsection{Stability of the second order scheme}
\label{sub:sosl}
Let us now obtain stability of the scheme \eqref{eq:velsl}--\eqref{eq:pressl} in the case $m=2$. To our knowledge, together with the scheme in \cite{MR2177143} these are the only unconditionally stable formally second order schemes for the Navier Stokes equations with traction boundary conditions. The idea of the proof is very similar to that of Theorem~\ref{thm:stabfosl}, yet it is inevitably obscured by tedious technical calculations that are necessary to properly balance all the terms. \AJS{These calculations are modifications of \cite[\S\S 5.2--5.3]{GuermondSalgadoErrorAnalysis}. In particular, in the last step, instead of a discrete Gr\"onwall's inequality we will employ the three-term recursion inequalities of Lemma~\ref{lem:threeterm}}.

\begin{thm}[Stability for $m=2$]
\label{thm:stabsosl}
If $m=2$ {and $\dt$ is sufficently small}, then the scheme \eqref{eq:velsl}--\eqref{eq:pressl} is stable provided that $\alpha > \max\{1,2/\Reyn\}$, in the sense that the solution satisfies
\[
  \| \bu_h^\dt \|_{\ell^\infty(\bL^2)} + \AJS{\| \calA(\bu_h^\dt) \|_{\ell^2(\Real)}} \leq C,
\]
where $C$ is a constant that does not depend on the solution of the scheme nor the discretization parameters. The constant $C$, however, might depend on $\Reyn$.
\end{thm}
\begin{proof}
First we set $\bv_h = 4\dt \bu_h^{k+1}$ in \eqref{eq:velsl} then  multiply \eqref{eq:pressl}, at time $t=t_k$, by $4\dt \Pi_{Q_h} \DIV \bu_h^{k+1}$. Next, we also set $r_h = 4\tfrac\dt\Reyn q^k$ in \eqref{eq:divsl} and  $z_h = \tfrac{8\dt^2}3( \psi_h^k + \psi_h^{\sharp,k})$ in \eqref{eq:projsl}. \AJS{Add the resulting equations and employ the identity
\[
  2\dt a^{k+1} \bdf{a^{k+1}}{2} = 2\dt \bdf{ |a^{k+1}|^2}{2} + 2 \frakd |\frakd a^{k+1}|^2 + |\frakd^2 a^{k+1}|^2
\]
to obtain}
\begin{multline}
\label{eq:so1sl}
  2\dt \bdf{ \| \bu_h^{k+1} \|^2 + \alpha \| \DIV \bu_h^{k+1} \|^2 }{2} + 2 \frakd \left( \| \frakd \bu_h^{k+1} \|^2 + \alpha \| \DIV \frakd \bu_h^{k+1} \|^2 + \frac{\dt}\Reyn \| q_h^{k+1}\|^2 \right) \\
  + \| \frakd^2 \bu_h^{k+1} \|^2 + \alpha \| \DIV \frakd^2 \bu_h^{k+1} \|^2
  + \AJS{4\dt \calA(\bu_h^{k+1})^2} + \frac{8\dt^2}3 \sbl \frakd \psi_h^{k+1}, \psi_h^k + \psi_h^{\sharp,k} \sbr
  = \AJS{2}\frac\dt\Reyn \| \frakd q_h^{k+1} \|^2.
\end{multline}
We now proceed as follows. First, notice that
\[
  \| \DIV \frakd \bu_h^{k+1} \|^2 = \left\| \DIV \frakd \bu_h^{k+1} + \frac{2\dt}3 \frakd^2 \psi_h^{k+1} \right\|^2 + \frac{4\dt^2}9 \|\frakd^2 \psi_h^{k+1} \|^2
  -\frac{4\dt}3 \scl \frakd^2 \psi_h^{k+1}, \DIV \frakd \bu_h^{k+1} \scr - \frac{8\dt^2}9 \| \frakd^2 \psi_h^{k+1} \|^2.
\]
Apply $\frakd$ to \eqref{eq:projsl} and set $z_h = \frakd^2 \psi_h^{k+1}$ to obtain 
\[
  -\frac{8\dt^2}9 \| \GRAD \frakd^2 \psi_h^{k+1} \|^2 = \frac{4\dt}3 \scl \DIV \frakd \bu_h^{k+1}, \AJS{ \frakd^2 \psi_h^{k+1} } \scr + \frac{8\dt^2}9 \| \frakd^2 \psi_h^{k+1} \|^2.
\]
Therefore we get
\begin{equation}
\label{eq:so2sl}
  \| \DIV \frakd\bu_h^{k+1} \|^2 = \left\| \DIV \frakd \bu_h^{k+1} + \frac{2\dt}3 \frakd^2 \psi_h^{k+1} \right\|^2 + \frac{4\dt^2}9 \trnorm{\frakd^2 \psi_h^{k+1} }^2 + \frac{4\dt^2}9 \|\GRAD \frakd^2 \psi_h^{k+1} \|^2.
\end{equation}
Insert this identity in \eqref{eq:so1sl} to conclude {that}
\begin{multline}
\label{eq:so3sl}
  2\dt \bdf{ \| \bu_h^{k+1} \|^2 + \alpha \| \DIV \bu_h^{k+1} \|^2 }{2} \\
  + 2 \frakd \left( \| \frakd \bu_h^{k+1} \|^2 + \alpha \left\| \DIV \frakd \bu_h^{k+1} + \frac{2\dt}3 \frakd^2 \psi_h^{k+1} \right\|^2
  + \frac{\AJS{4}\alpha\dt^2}9 \|\GRAD \frakd^2 \psi_h^{k+1} \|^2  + \frac{\AJS{4}(\alpha-1)\dt^2}9 \trnorm{\frakd^2 \psi_h^{k+1} }^2
  + \frac{\dt}\Reyn \| q_h^{k+1}\|^2 \right) \\
  + \| \frakd^2 \bu_h^{k+1} \|^2 + \alpha \| \DIV \frakd^2 \bu_h^{k+1} \|^2
  + \AJS{4\dt \calA(\bu_h^{k+1})^2} 
  + \frac{8\dt^2}9 \frakd \trnorm{\frakd^2 \psi_h^{k+1} }^2 + \frac{8\dt^2}3 \sbl \frakd \psi_h^{k+1}, \psi_h^k + \psi_h^{\sharp,k} \sbr
  = \AJS{2}\frac\dt\Reyn \| \frakd q_h^{k+1} \|^2.
\end{multline}
Since, by assumption, $\alpha >1$, it is now necessary to control the last term on the left hand side of \eqref{eq:so3sl}. To do so, we begin by writing
\begin{align*}
  \frac{8\dt^2}3 \sbl \frakd \psi_h^{k+1}, \psi_h^k + \psi_h^{\sharp,k} \sbr &=
  \frac{8\dt^2}3 \left( \sbl \frakd \psi_h^{k+1}, \psi_h^k \sbr \AJS{+} \trnorm{\frakd \psi_h^{k+1}}^2 \AJS{-} \sbl \frakd\psi_h^{k+1}, \frakd^2 \psi_h^{k+1} \sbr 
  + \frac13 \sbl \frakd \psi_h^{k+1}, \frakd^2 \psi_h^{\AJS{k}} \sbr \right)
  \\
  &=  \frac{4\dt^2}3 \left( \frakd \trnorm{ \psi_h^{k+1} }^2 + \trnorm{\frakd \psi_h^k }^2 - \trnorm{ \frakd^2 \psi_h^{k+1} }^2 \right) + \frac{8\dt^2}9 \sbl \frakd \psi_h^{k+1}, \frakd^2 \psi_h^k \sbr.
\end{align*}
Next, we notice that
\begin{multline*}
  \frac{8\dt^2}9 \frakd \trnorm{ \frakd^2 \psi_h^{k+1} }^2 - \frac{4\dt^2}3\trnorm{ \frakd^2 \psi_h^{k+1} }^2 
  + \frac{8\dt^2}9 \sbl \frakd \psi_h^{k+1}, \frakd^2 \psi_h^k \sbr = \\
  \frac{8\dt^2}9 \sbl \frakd \psi_h^k, \frakd^2 \psi_h^k \sbr - \frac{4\dt^2}9 
  \left( \trnorm{ \frakd^2 \psi_h^{k+1} }^2 + 2\trnorm{ \frakd^2 \psi_h^k }^2 - 2 \sbl \frakd^2 \psi_h^{k+1}, \frakd^2 \psi_h^k \sbr \right)
  = \frac{4\dt^2}9 \left( \frakd \trnorm{ \frakd \psi_h^k }^2 - \trnorm{ \frakd^3 \psi_h^{k+1} }^2 \right).
\end{multline*}
In conclusion,
\begin{multline}
\label{eq:so4sl_identity}
  \frac{8\dt^2}9 \frakd \trnorm{\frakd^2 \psi_h^{k+1} }^2 + \frac{8\dt^2}3 \sbl \frakd \psi_h^{k+1}, \psi_h^k + \psi_h^{\sharp,k} \sbr \\ =
  \frac{4\dt^2}3 \left( \frakd \trnorm{ \psi_h^{k+1} }^2 + \trnorm{\frakd \psi_h^k }^2 \right) 
  + \frac{4\dt^2}9 \left( \frakd \trnorm{ \frakd \psi_h^k }^2 - \trnorm{ \frakd^3 \psi_h^{k+1} }^2 \right).
\end{multline}
Substitute \eqref{eq:so4sl_identity} in \eqref{eq:so3sl} to obtain
\begin{multline}
\label{eq:so4sl}
  2\dt \bdf{ \| \bu_h^{k+1} \|^2 + \alpha \| \DIV \bu_h^{k+1} \|^2 }{2} \\
  + 2 \frakd \left( \| \frakd \bu_h^{k+1} \|^2 + \alpha \left\| \DIV \frakd \bu_h^{k+1} + \frac{2\dt}3 \frakd^2 \psi_h^{k+1} \right\|^2
  + \frac{\AJS{4}\alpha\dt^2}9 \|\GRAD \frakd^2 \psi_h^{k+1} \|^2  + \frac{\AJS{4}(\alpha-1)\dt^2}9 \trnorm{\frakd^2 \psi_h^{k+1} }^2 \right. \\
  \left.
  + \frac{\dt}\Reyn \| q_h^{k+1}\|^2 + \frac{2\dt^2}9 \left( 3\trnorm{\psi_h^{k+1}}^2 + \trnorm{\frakd \psi_h^k}^2 \right)
  \right)
  + \frac{4\dt^2}3 \trnorm{\frakd \psi_h^k }^2
  + \| \frakd^2 \bu_h^{k+1} \|^2 + \alpha \| \DIV \frakd^2 \bu_h^{k+1} \|^2 \\
  + \AJS{4\dt \calA(\bu_h^{k+1})^2}
  = \AJS{2}\frac\dt\Reyn \| \frakd q_h^{k+1} \|^2 + \frac{4\dt^2}9 \trnorm{ \frakd^3 \psi_h^{k+1} }^2.
\end{multline}
It remains then to control the terms on the right hand side, which is obtained as in the first order case. Apply $\frakd^2$ to \eqref{eq:projsl} and set $z_h = \frakd^3 \psi_h^{k+1}$ to obtain
\[
  \frac{4\dt^2}9 \trnorm{ \frakd^3 \psi_h^{k+1} }^2 \leq \| \DIV \frakd^2 \bu_h^{k+1} \|^2.
\]
Equation \eqref{eq:divsl} implies
\[
  \| \frakd q_h^{k+1} \|^2 \leq \| \DIV \bu_h^{k+1} \|^{{2}}.
\]
\shlee{Finally, by inserting these bounds in \eqref{eq:so4sl} \shlee{and} using that $\alpha > \max\{1,2/\Reyn\}$,
we can apply the three term recursion inequalities of Lemma~\ref{lem:threeterm}.} \AJS{Indeed, setting
\begin{align*}
  a^k &= \| \bu_h^k \|^2 + \alpha \| \DIV \bu_h^k \|^2, \
  b^k = 4 \dt \calA(\bu_h^k)^2, \\
  d^k &= 2 \left( \| \frakd \bu_h^k \|^2 + \alpha \left\| \DIV \frakd \bu_h^k + \frac{2\dt}3 \frakd^2 \psi_h^k \right\|^2
  + \frac{4\alpha\dt^2}9 \|\GRAD \frakd^2 \psi_h^k \|^2  + \frac{4(\alpha-1)\dt^2}9 \trnorm{\frakd^2 \psi_h^k }^2 \right. \\
 &\left.  \hspace*{2.5in} 
  + \frac{\dt}\Reyn \| q_h^k\|^2 + \frac{2\dt^2}9 \left( 3\trnorm{\psi_h^k}^2 + \trnorm{\frakd \psi_h^{k-1}}^2 \right)
  \right), \\
  A &= 3(1-\dt) >0, \
  B = -4 < 0,  \
  C = 1 >0,
\end{align*}
we observe that our previous discussion implies
\[
  A a^{k+1} B a^k + C a^{k-1} \leq - (b^{k+1}+d^{k+1}-d^k).
\]
In addition, $A+B+C = -3\dt<0$ and, if $\dt$ is small enough, the equation $Ar^2+Br+C=0$ has roots
\begin{align*}
  r_1 &= \frac{2-\sqrt{1-3\dt}}{3(1-\dt)} = \frac13 \left( 1-\frac\dt2 + \calO(\dt^2) \right), \\
  r_2 &= \frac{2+\sqrt{1-3\dt}}{3(1-\dt)} = 1 + \frac{3\dt}2 + \calO(\dt^2).
\end{align*}
Both roots are positive; $r_1 <\tfrac13$ and $r_2$ is larger but close to one. Consequently, Lemma~\ref{lem:threeterm} implies that, for $n\geq2$, we have
\[
  a^n \leq C(a^0+a^1)(r_1^n + r_2^n) - \frac1{3(1-\dt)}\sum_{k=2}^n r_1^{n-k}\sum_{s=2}^k r_2^{k-s}(b^s +d^s-d^{s-1}),
\]
which, since $\dt $ is small can be rewritten as
\[
  a^n + \frac13 \sum_{k=2}^n r_1^{n-k} \sum_{s=2}^k r_2^{k-s} b^s \leq C_1(1+\exp(C_2T))(a^0+a^1) - \frac1{3(1-\dt)} \sum_{k=2}^n r_1^{n-k}\sum_{s=2}^k r_2^{k-s}(d^s-d^{s-1}),
\]
for some constants $C_1$ and $C_2$. To handle the last term we argue as in \cite[Theorem 5.2]{GuermondSalgadoErrorAnalysis}. This implies the result.
}
\end{proof}
Notice that a similar observation to Remark~\ref{rem:newpresup} is also valid here.

\section{A scheme with modification of the boundary correction}
\label{sec:Baensch}

Here we present a gauge Uzawa method for the boundary correction scheme discussed in \cite{MR3253466}. 
We show stability for first and second order variants of this method if the time step and mesh size satisfy a certain condition.
We seek for sequences $\bu_h^\dt \subset \bX_h$, $\psi_h^\dt \subset M_h$, $q_h^\dt, p_h^\dt \subset Q_h$, where $\bu_h^\dt$ and $p_h^\dt$ are used to approximate the velocity and pressure, respectively.
Note that here $\psi_h^\dt \subset M_h$. This is in contrast to what was adopted in Section~\ref{sec:Sanghyun}, \ie $\psi_h^\dt \subset Q_h$.
 After an initialization as in \eqref{eq:init} the scheme proceeds, for $k=0,\ldots,\calK-1$, as follows:
\begin{enumerate}[$\bullet$]
  \item \textbf{Velocity update:} Find $\bu_h^{k+1} \in \bX_h$ that satisfies, for all $\bv_h \in \bX_h$,
  \begin{multline}
  \label{eq:velstep}
    \scl \bdf{\bu_h^{k+1}}{m}, \bv_h \scr + \AJS{\calA \scl \bu_h^{k+1}, \bv_h \scr} - \scl p_h^k, \DIV \bv_h \scr 
    + \scl \GRAD \psi_h^{\sharp,k}, \bv_h \scr = \\
    \scl \bef^{k+1}, \bv_h \scr - \frac\dt{\beta\shlee{\Reyn}}  \scl \partial_\bn \frakd \psi_h^k, \DIVG \bv_h \scr_\Gamma + \scl \bg^{k+1}, \bv_h \scr_\Gamma,
  \end{multline}
  where $\bdf{\cdot}{m}$, \AJS{$\calA\scl\cdot,\cdot\scr$} and $(\cdot)^{\sharp,k}$ where defined in \eqref{eq:defofbdf}, \AJS{\eqref{eq:defofcalA}} and \eqref{eq:defofsharp}, respectively. As before, the parameter $\beta$ is set to $\beta = 1 + \tfrac12(m -1)$. By $\DIVG \bv_h$ we denote the surface divergence of $\bv_h$.
  
  \item \textbf{Projection:} Find $\psi_h^{k+1} \in M_h$ that satisfies
  \begin{equation}
  \label{eq:projstep}
    \scl \GRAD \frakd \psi_h^{k+1}, \GRAD z_h \scr = \frac\beta\dt \scl \bu_h^{k+1}, \GRAD z_h \scr, \quad \forall z_h \in M_h.
  \end{equation}
  
  \item \textbf{Divergence correction:} Find $q_h^{k+1} \in Q_h$ that solves
  \begin{equation}
  \label{eq:divstep}
    \scl \frakd q_h^{k+1}, r_h \scr = - \scl \DIV \bu_h^{k+1}, r_h \scr, \quad \forall r_h \in Q_h.
  \end{equation}
  
  \item \textbf{Pressure update:} The new pressure $p_h^{k+1} \in Q_h$ is obtained by setting
  \begin{equation}
  \label{eq:presup}
    p_h^{k+1} = \psi_h^{k+1} + \frac\kappa\Reyn q_h^{k+1},
  \end{equation}
  where $\kappa$ was defined in \eqref{eq:Korn}.

\end{enumerate}

\subsection{The scheme as a rotational pressure correction method}
\label{sub:ourscheme}

With the result of Proposition~\ref{prop:equiv} at hand it is easy to provide a motivation for scheme \eqref{eq:velstep}--\eqref{eq:presup}. Indeed, if we were able to integrate back by parts the momentum equation, we would obtain
\begin{multline*}
  \scl \bdf{\bu^{k+1}}{m} - \DIV\left( \frac{\AJS{2}}\Reyn \vare(\bu_h^{k+1}) \right) + \GRAD\left( p_h^k + \psi_h^{\sharp,k} \right), \bv_h \scr
  + \scl \frac{\AJS{2}}\Reyn \vare(\bu_h^{k+1} )\SCAL \bn - p_h^k \bn, \bv_h \scr_\Gamma = \\
  \scl \bef^{k+1}, \bv_h \scr + \scl \bg^{k+1} + \frakL_h^{k+1}, \bv_h \scr_\Gamma,
\end{multline*}
where, as in \cite{MR3253466}, we introduced $\frakL_h^\dt$ to be the solution of
\[
  \scl \frakL_h^{k+1}, \bv_h \scr_\Gamma = - \frac\dt{\beta\shlee{\Reyn}} \scl \partial_\bn \frakd \psi_h^k, \DIVG \bv_h \scr_\Gamma.
\]
This would imply that
\[
  \bdf{\bu^{k+1}}{m} - \DIV\left( \frac{\AJS{2}}\Reyn \vare(\bu_h^{k+1}) \right) + \GRAD\left( p_h^k + \psi_h^{\sharp,k} \right) = \bef^{k+1},
\]
and
\[
  \left( \frac{\AJS{2}}\Reyn \vare(\bu_h^{k+1} ) - p_h^k \polI \right)_{|\Gamma} \SCAL \bn = \bg^{k+1} + \frakL_h^{k+1},
\]
as in \cite{MR3253466}. Finally, integrating back by parts \eqref{eq:projstep} would yield
\[
  - \scl \LAP \frakd \psi_h^{k+1}, z_h \scr = -\frac\beta\dt \scl \DIV \bu_h^{k+1}, z_h \scr
\]
or
\[
  \LAP \frakd \psi_h^{k+1} = \frac\beta\dt \DIV \bu_h^{k+1}.
\]
Notice that these coincide with the equations of a stabilized gauge Uzawa scheme. 
Then by using the equivalence given in Proposition~\ref{prop:equiv}, we conclude that \eqref{eq:velstep}--\eqref{eq:presup} is the same as the scheme of \cite{MR3253466}, written in a slightly different form.

\subsection{Stability analysis of the first order scheme}
\label{sec:fo}

Here we present a stability analysis of scheme \eqref{eq:velstep}--\eqref{eq:presup} for $m=1$ (first order variant). As in the case of the grad-div stabilized scheme of Section~\ref{sec:Sanghyun}, our real interest is in $m=2$, but we present this because the arguments are simpler and will allow us to clarify the discussion in the analysis of the second order variant provided below.
To avoid irrelevant technicalities, assume that $\bef \equiv \boldsymbol{0}$ and $\bg \equiv \boldsymbol{0}$.

\begin{thm}[Stability for $m=1$]
\label{thm:stabfo}
Assume that the space $M_h$ verifies \eqref{eq:invineq}. If $m=1$, {$\dt$ is sufficiently small} and the mesh size and time step satisfy
\begin{equation}
\label{eq:CFL}
  \tau \leq C \Reyn h^2
\end{equation}
then the scheme \eqref{eq:velstep}--\eqref{eq:presup} is stable, in the sense that it satisfies
\[
  \| \bu_h^\dt \|_{\ell^\infty(\bL^2)} + \AJS{\| \calA(\bu_h^\dt) \|_{\ell^2(\Real)}} \leq C,
\]
where the constant might depend on $\Reyn$, but is independent of the solution of the scheme or discretization parameters.
\end{thm}
\begin{proof}
Set $\bv_h = 2\dt \bu_h^{k+1}$ in \eqref{eq:velstep};
multiply \eqref{eq:presup}, at time $t=t_k$, by $-2\dt \Pi_{Q_h} \DIV \bu_h^{k+1}$ and integrate over $\Omega$;
finally, set $r_h = 2 \tfrac{\kappa\dt}\Reyn q_h^k$ in \eqref{eq:divstep}.
Adding the ensuing identities we obtain
\begin{equation}
\label{eq:almostfo}
  \frakd \| \bu_h^{k+\shlee{1}}\|^2 + \| \frakd \bu_h^{k+1} \|^2 + \AJS{2\dt \calA(\bu_h^{k+1})^2} + \frac{\kappa\dt}\Reyn \frakd \| q_h^{k+1} \|^2 +
  2\dt \scl \GRAD( \psi_h^k + \frakd \psi_h^{\AJS{k}} ), \bu_h^{k+1} \scr
  \leq 
  2\dt \scl \frakL_h^{k+1}, \bu_h^{k+1} \scr_\Gamma + \frac{\kappa \dt }\Reyn \|\DIV \bu^{k+1} \|^2,
\end{equation}
where we used that $\psi_h^\dt \in M_h \subset \Hunz$ to infer
\[
  \scl \GRAD \frakd \psi_h^k, \bu_h^{\AJS{k+1}} \scr = - \scl \frakd \psi_h^k, \DIV \bu_h^{k+1} \scr,
\]
{and that, setting $r_h = \frakd q_h^{k+1}$}, in \eqref{eq:divstep} reveals 
\[
  \| \frakd q_h^{k+1} \| \leq \| \DIV \bu_h^{k+1} \|.
\]
Next we set $z_h = 2\dt^2( \psi_h^k + \psi_h^{\sharp,k})$ in \eqref{eq:projstep} to get
\begin{equation}
\label{eq:andonefo}
  \dt^2 \left( \frakd \| \GRAD \psi_h^{k+1} \|^2 + \| \GRAD \frakd \psi_h^{\AJS{k}} \|^2 - \| \GRAD \frakd^2 \psi_h^{k+1}\|^2 \right)
  =
  2\dt \scl\bu_h^{k+1}, \GRAD( \psi_h^k + \frakd \psi_h^k ) \scr.
\end{equation}
Applying $\frakd$ to \eqref{eq:projstep} and setting $z_h = \frakd^2 \psi_h^{k+1}$ yields
\begin{equation}
\label{eq:andtwofo}
  \dt^2 \| \GRAD \frakd^2 \psi_h^{k+1} \|^2 \leq \| \frakd \bu_h^{k+1} \|^2.
\end{equation}
Now we add \eqref{eq:almostfo}, \eqref{eq:andonefo} and \eqref{eq:andtwofo} to derive
\begin{equation}
\label{eq:priortoinvineqs}
  \frakd \| \bu_h^{k+\shlee{1}}\|^2 + \AJS{2\dt \calA(\bu_h^{k+1})^2} + \frac{\kappa\dt}\Reyn \frakd \| q_h^{k+1} \|^2 +
  \dt^2 \left( \frakd \| \GRAD \psi_h^{k+1} \|^2 + \| \GRAD \frakd \psi_h^{\AJS{k}} \|^2 \right)
  \leq 
  2\dt \scl \frakL_h^{k+1}, \bu_h^{k+1} \scr_\Gamma + \frac{\kappa \dt }\Reyn \|\DIV \bu^{k+1} \|^2.
\end{equation}
Notice that {inequality \eqref{eq:Korn} implies}
\[
  \dfrac{\kappa\tau}{\Reyn} \| \DIV \bu_h^{k+1} \|^{{2}} 
  \leq \dfrac{\tau}{\Reyn} \| \bu_h^{k+1} \|^2 + \tau{ \calA(\bu_h^{k+1} )^2}.
\]
Moreover, using \cite[Lemma 3.3]{MR3253466}, we see that
\[
  2\dt \scl \frakL_h^{k+1},\bu_h^{k+1} \scr_\Gamma = 
  \shlee{- \dfrac{2\dt^2}{\beta \Reyn}} \scl \partial_n \frakd \psi_h^k, \DIVG \bu_h^{k+1} \scr_\Gamma
  \leq \shlee{ \dfrac{C \dt^2}{\Reyn}} \| \bu_h^{k+1} \|_{1/2,\Gamma} \| \partial_\bn \frakd \psi_h^k \|_{1/2,\Gamma}.
\]
Using the trace inequality \eqref{eq:trace}, inequality \eqref{eq:Korn} and the inverse estimate \eqref{eq:invineq} 
we continue this bound as follows:
\begin{align*}
  2\dt \scl \frakL_h^{k+1},\bu_h^{k+1} \scr_\Gamma  
  &\leq C \shlee{\dfrac{\dt^2 h^{-1}}{\Reyn}} \|\GRAD \frakd \psi_h^k \| 
    \left( \| \bu_h^{k+1} \|^2 + \AJS{\Reyn \calA(\bu_h^{k+1} )^2} \right)^{1/2} \\
    &\leq \frac\dt{2\Reyn} \left( \| \bu_h^{k+1} \|^2 + \AJS{\Reyn \calA(\bu_h^{k+1} )^2} \right)
    + C \dfrac{\dt^3 }{\shlee{\Reyn h^{2}}} \| \GRAD \frakd \psi_h^k \|^2 \\
    &\leq \frac\dt{2\Reyn} \left( \| \bu_h^{k+1} \|^2 + \AJS{\Reyn \calA(\bu_h^{k+1} )^2} \right)
    + \dt^2 \| \GRAD \frakd \psi_h^k \|^2,
\end{align*}
where, in the last step, we used the mesh condition \eqref{eq:CFL}. Inserting these observations into \eqref{eq:priortoinvineqs} yields the final estimate
\[
  \left(1-\AJS{\frac\dt{2\Reyn}} \right) \| \bu_h^{k+1}\|^2 + \AJS{ \frac\dt2 \calA(\bu_h^{k+1})^2} + \frac{\kappa\dt}\Reyn \frakd \| q_h^{k+1} \|^2 + \dt^2 \frakd \|\GRAD \psi_h^{k+1}\|^2 \leq \| \bu_h^k\|^2.
\]
\AJS{We now proceed as in the proof of Theorem~\ref{thm:stabfosl} and use Lemma~\ref{lem:gronwall}. This concludes the proof}.
\end{proof}

\subsection{Stability analysis of the second order scheme}
\label{sec:so}

Let us now present the stability analysis for the second order scheme ($m = 2$). The proof combines the ideas of Theorem~\ref{thm:stabsosl} and Theorem~\ref{thm:stabfo}. For this reason we keep details to a minimum. The stability is as follows.

\begin{thm}[Stability for $m=2$]
\label{thm:stabso}
Assume that the space $M_h$ satisfies \eqref{eq:invineq}. If $m =2 $, {$\dt$ is sufficiently small} and \eqref{eq:CFL} holds, then the scheme \eqref{eq:velstep}--\eqref{eq:presup} is stable, in the sense that it satisfies
\[
  \| \bu_h^\dt \|_{\ell^\infty(\bL^2)} + \AJS{\| \calA(\bu_h^\dt) \|_{\ell^2(\Real)}} \leq C,
\]
where the constant $C$ might depend on $\Reyn$, but is independent of the solution of the scheme or discretization parameters.
\end{thm}
\begin{proof}
Set $\bv_h = 4\dt \bu_h^{k+1}$ in \eqref{eq:velstep};
multiply \eqref{eq:presup},  at time $t=t_k$, by $-4\dt \Pi_{Q_h} \DIV \bu_h^{k+1}$ and integrate over $\Omega$; set $r_h = 4\tfrac{\kappa\dt}\Reyn q_h^k$ in \eqref{eq:divstep} and $z_h = \tfrac{8\dt^2}3 \left( \psi_h^k -\psi_h^{\sharp,k} \right)$ in \eqref{eq:projstep}; apply the operator $\frakd$ to \eqref{eq:projstep} and set $z_h = \frakd^2 \psi_h^{k+1}$. Adding the ensuing identities yields, after tedious calculations which nevertheless closely follow the arguments of the proof of Theorem~\ref{thm:stabsosl}, that
\begin{multline}
\label{eq:ppso}
  2\dt \bdf{ \| \bu_h^{k+1} \|^2 }{2}
  +2 \frakd \left( \left\| \frakd \bu_h^{k+1} - \frac{2\dt}3 \GRAD \frakd^2 \psi_h^{k+1} \right\|^2 
  + \frac{\kappa\dt}\Reyn \| q_h^{k+1} \|^2 + \frac{2\dt^2}3 \| \GRAD \psi_h^{k+1} \|^2 + \frac{4\dt^2}9 \| \GRAD \frakd \psi_h^k \|^2
  \right) \\
  + \| \frakd^2 \bu_h^{k+1}\|^2 + \frac{4\dt^2}3 \| \GRAD \frakd \psi_h^k \|^2 
  + \AJS{ 4\dt \calA(\bu_h^{k+1})^2} 
  \leq
  \frac{2\kappa\dt}\Reyn \| \DIV \bu_h^{k+1} \|^2 + 4\dt \scl \frakL_h^{k+1},\bu_h^{k+1} \scr_\Gamma + \frac{4\dt^2}9 \| \GRAD \frakd^3 \psi_h^{k+1} \|^2.
\end{multline}

Next, we apply $\frakd^2$ to \eqref{eq:projstep} and set $z_h = \frakd^3 \psi_h^{k+1}$ to get
\begin{equation}
\label{eq:ppso_2}
  \frac{4\dt^2}9 \| \GRAD \frakd^3 \psi_h^{k+1} \|^2 \leq \| \frakd^2 \bu_h^{k+1} \|^2.
\end{equation}

Adding \eqref{eq:ppso_2} and \eqref{eq:ppso} yields a suitable bound for the last term on the right hand side. The remaining terms can be handled as in the proof of Theorem~\ref{thm:stabfo}. \AJS{To conclude we follow the proof of Theorem~\ref{thm:stabsosl} and apply the three term recursion inequalities of Lemma~\ref{lem:threeterm}}.
\end{proof}

\begin{rem}[Other schemes]
The technique used in \cite[\S~4]{MR3327983} to show unconditional stability of their scheme is very similar to the ones we have discussed here. Thus, one can combine our ideas with their techniques to obtain unconditionally stable higher order schemes for the modified boundary condtions of \cite{MR3327983}. To avoid repetition, we skip these details.
\end{rem}

\section{Numerical illustrations}
\label{sec:numerics}
Let us, in this last section, evaluate the performance of the numerical schemes for traction boundary conditions discussed in previous sections. In \S\ref{subsec:grad_div}, we explore computationally the rate of convergence for the grad-div stabilized scheme presented and analyzed in Section~\ref{sec:Sanghyun} for $m=2$ (second order scheme). For the boundary correction scheme of Section~\ref{sec:Baensch}, similar computations are carried out in \S\ref{subsec:bd_correction}. 

All examples are computed with the help of the open-source finite element library \texttt{deal.II} \cite{dealII83}. 
In particular, the implementation is an extension of the framework used in \cite{MR3337644}.
We use the lowest order Taylor-Hood elements over quadrilateral meshes, that is $\polQ_2/\polQ_1$ finite elements. In all the experiments, the arising linear systems are solved using the generalized
minimal residual method (GMRES) solver with an AMG preconditioner.

For all convergence tests we set $\Omega = (0,1)^2$
and choose the right hand sides $\bef$ and $\bg$ so that the exact solution to \eqref{eq:NSE}--\eqref{eq:BC} is
\[
  \ue(t,x,y) := \left( 
        \begin{array}{c}
        \sin (t+x) \sin (t+y) \\
        \cos (t+x) \cos (t+y)
        \end{array}
  \right), 
  \qquad
  \pe(t,x,y) :=  \sin (t+x-y).
\]

\subsection{Grad-div stabilization}
\label{subsec:grad_div}

\begin{figure}[!h]
\centering
\begin{subfigure}[b]{0.46\textwidth}
    \includegraphics[width=\textwidth]{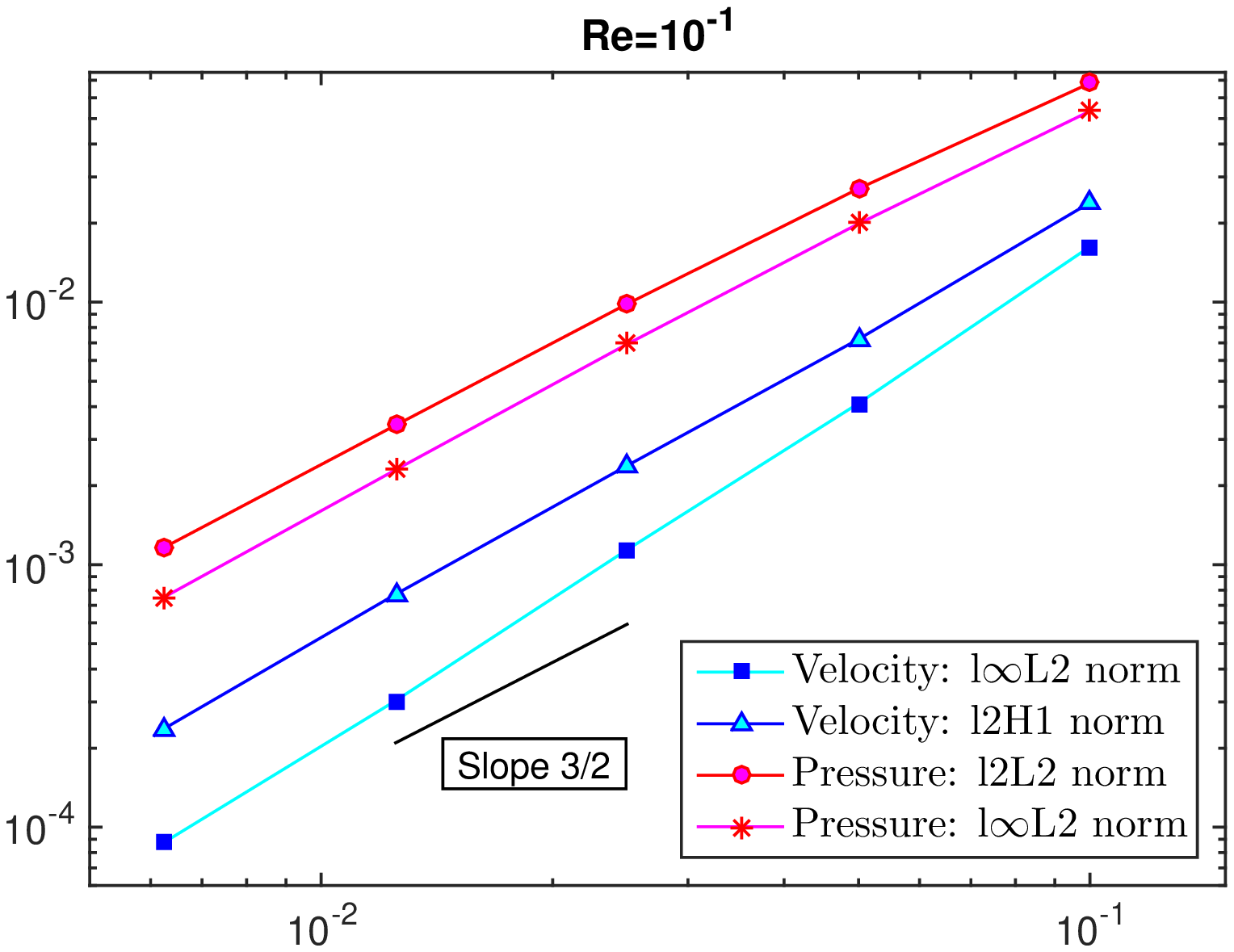}
\caption{$\Reyn=10^{-1}$}
\end{subfigure}
\begin{subfigure}[b]{0.46\textwidth}
    \includegraphics[width=\textwidth]{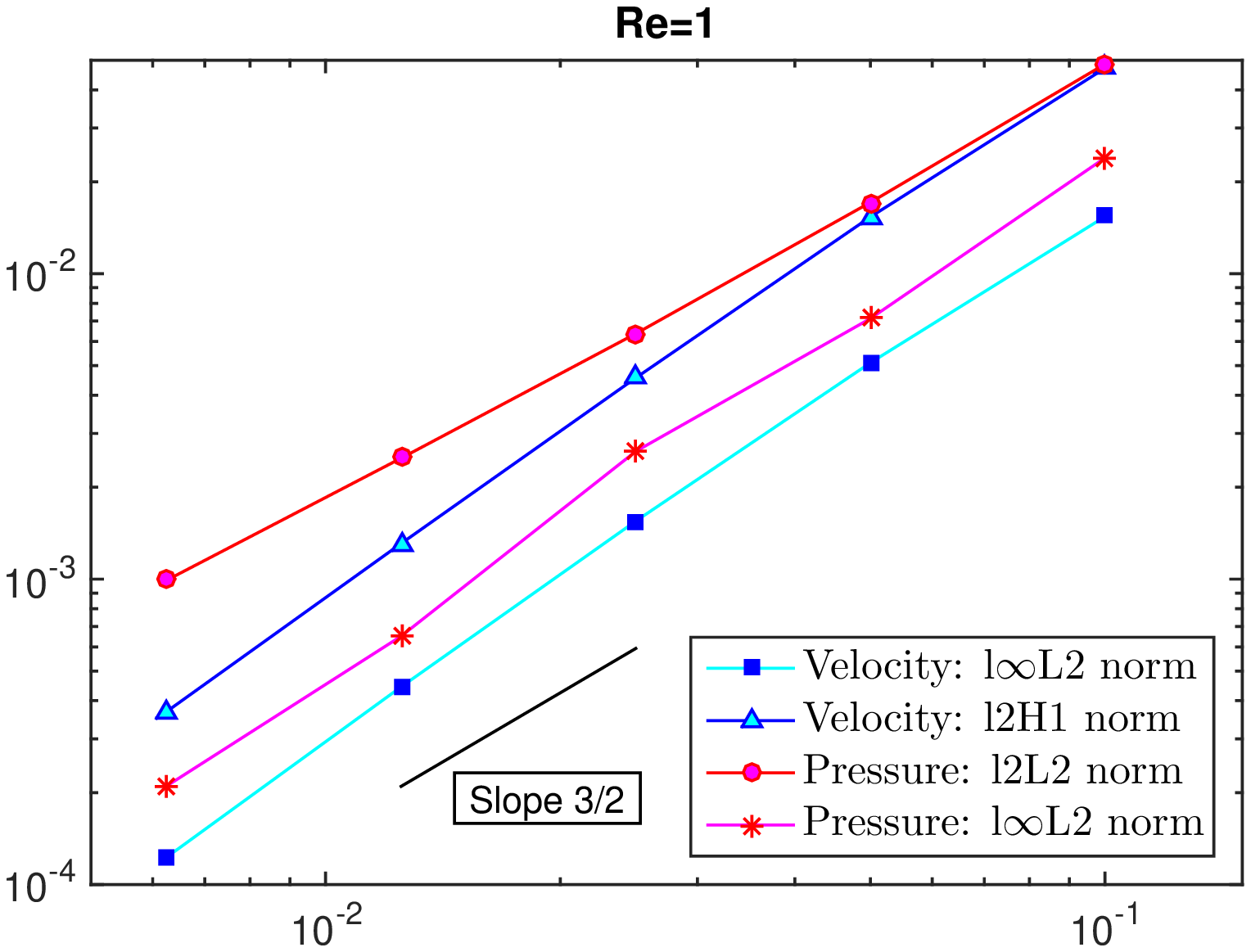}
\caption{$\Reyn=1$}
\end{subfigure}
\begin{subfigure}[b]{0.46\textwidth}
    \includegraphics[width=\textwidth]{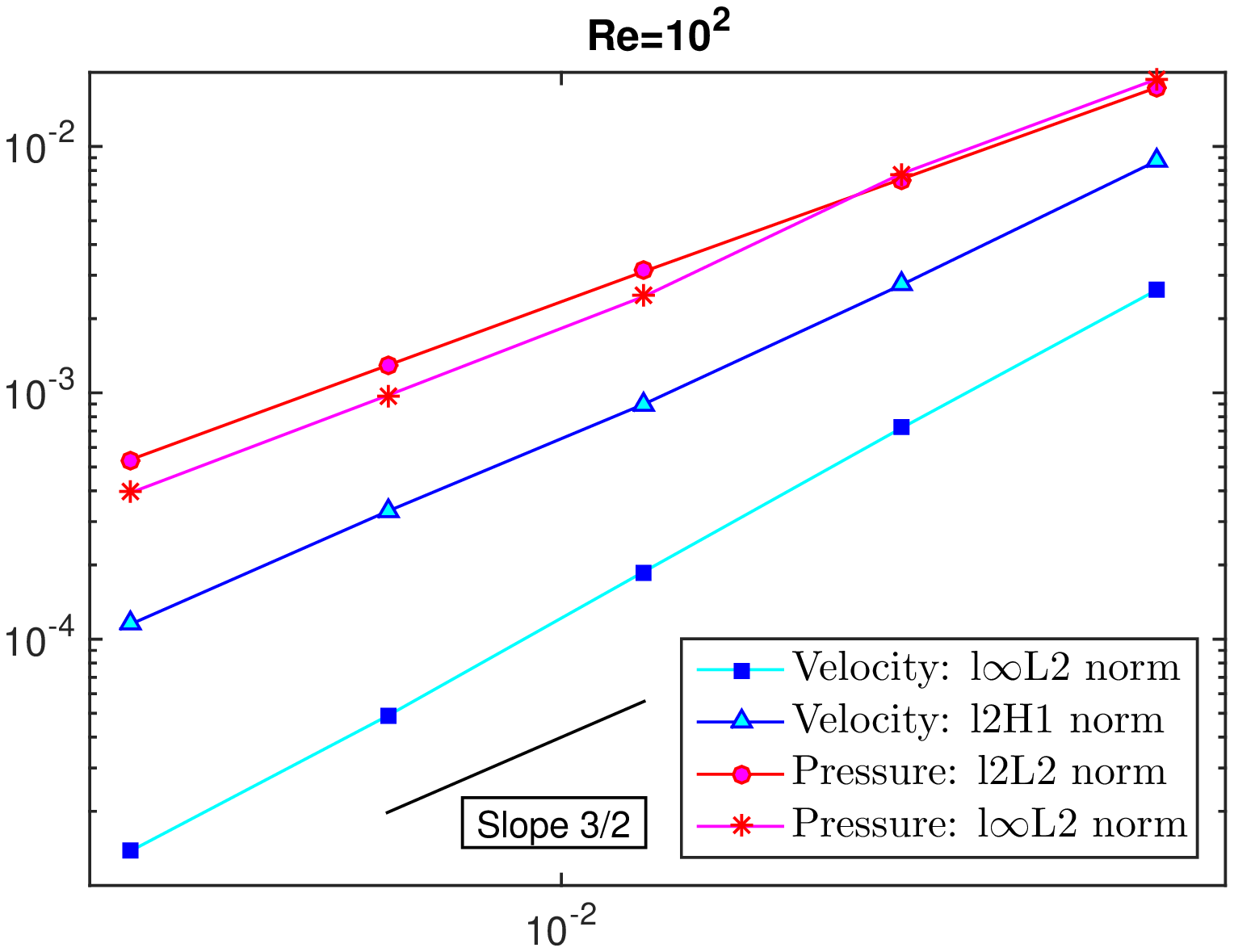}
\caption{$\Reyn=10^2$}
\end{subfigure}
\begin{subfigure}[b]{0.46\textwidth}
    \includegraphics[width=\textwidth]{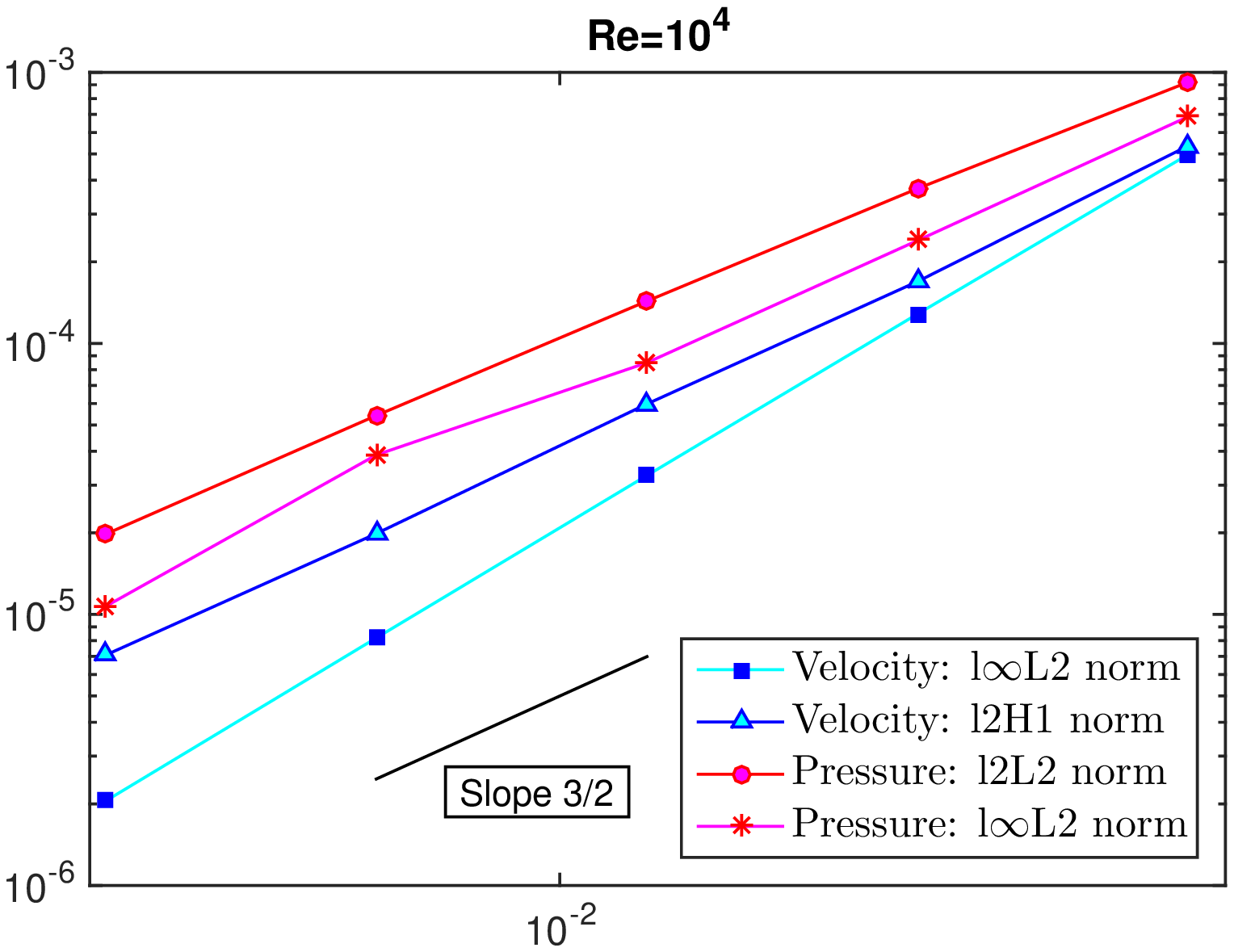}
\caption{$\Reyn=10^4$}
\end{subfigure}
  \caption{Decay of different error norms versus time step $\dt$ for the rotational pressure correction projection method with grad-div stabilization of Section~\ref{sec:Sanghyun} with different $\Reyn$ numbers. The order of convergence $\calO(\dt^{3/2})$ is observed as expected for all cases.} 
\label{fig:grad_div} 
\end{figure}

Here we explore the scheme \eqref{eq:init}--\eqref{eq:pressl} of Section~\ref{sec:Sanghyun}. Computational results for a first order discretization ($m=1$) and comparisons with suboptimal classical standard pressure correction schemes have been provided in \cite{SLee2013}. Therefore, here we focus on the second order ($m=2$) case. For all values of the discretization parameters we set $\alpha = 1$. The behavior of the errors in the velocity and pressure approximations versus the time step $\dt$ are depicted in Figure~\ref{fig:grad_div} \shlee{with different 
$\Reyn \in \{10^{-1},10^0,10^2,10^4 \}$  numbers.} 
The space discretization is chosen fine enough ($h=0.015625$) so that it does not pollute the time discretization error. 
However, note that larger mesh size {($h=0.0625$)} is chosen for the case  $\Reyn = 10^4$ to be efficient in computational time for solver.
We observe a rate of convergence of $\calO(\dt^{3/2})$ for the velocity in the $\ell^2(\bH^1(\Omega))$ and $\ell^{\infty}(\bL^2(\Omega))$ norms and the pressure in the  $\ell^2(L^2(\Omega))$ and $\ell^{\infty}(L^2(\Omega))$ norms for all cases.
Notice that this is to be expected, since this is the provable rate of convergence for the rotational scheme even in the case of no-slip boundary conditions \cite{Gue.S02_b}.


\subsection{Boundary correction scheme}
\label{subsec:bd_correction}

\begin{figure}[!h]
\centering
\begin{subfigure}[b]{0.46\textwidth}
    \includegraphics[width=\textwidth]{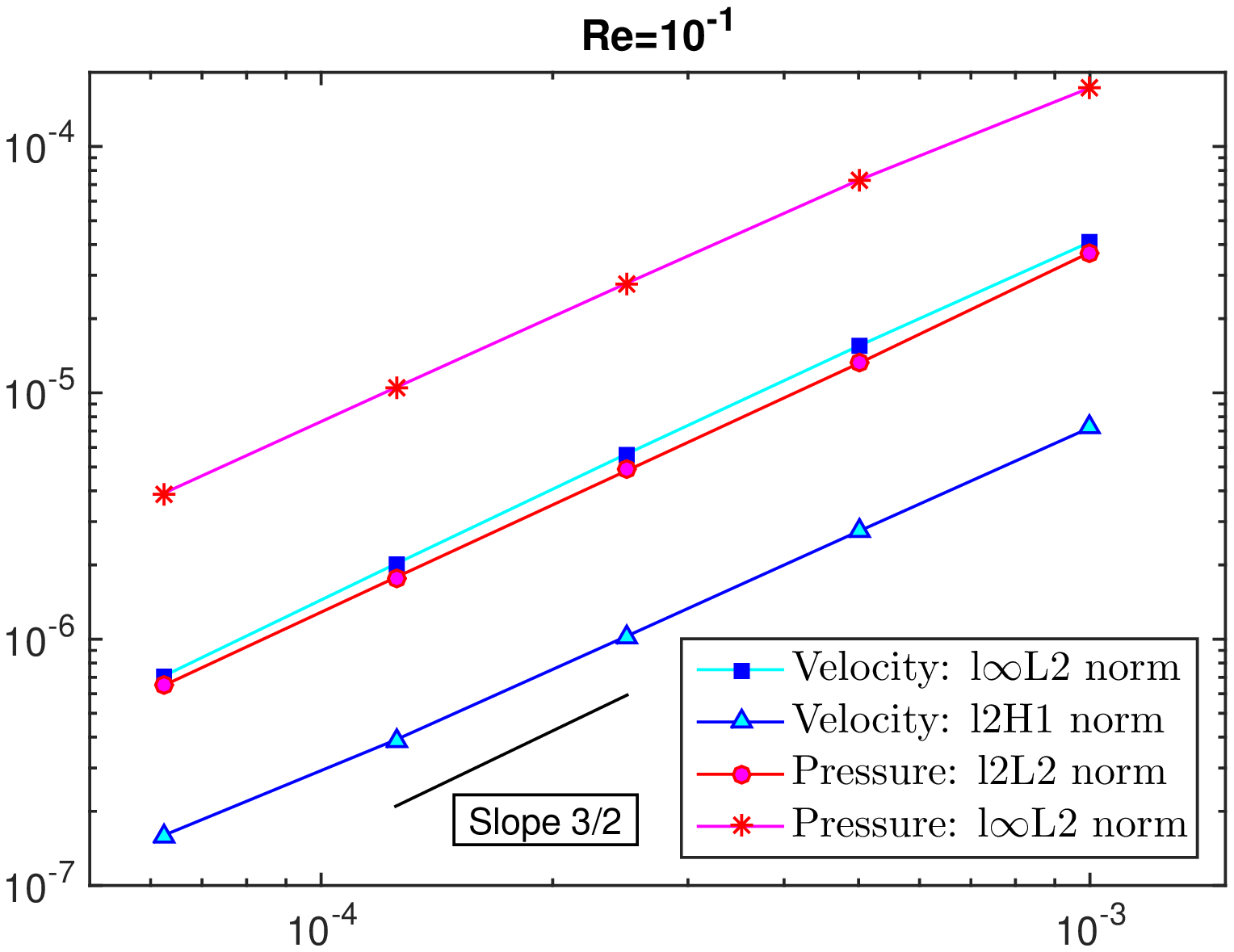}
\caption{$\Reyn=10^{-1}$}
\end{subfigure}
\begin{subfigure}[b]{0.46\textwidth}
    \includegraphics[width=\textwidth]{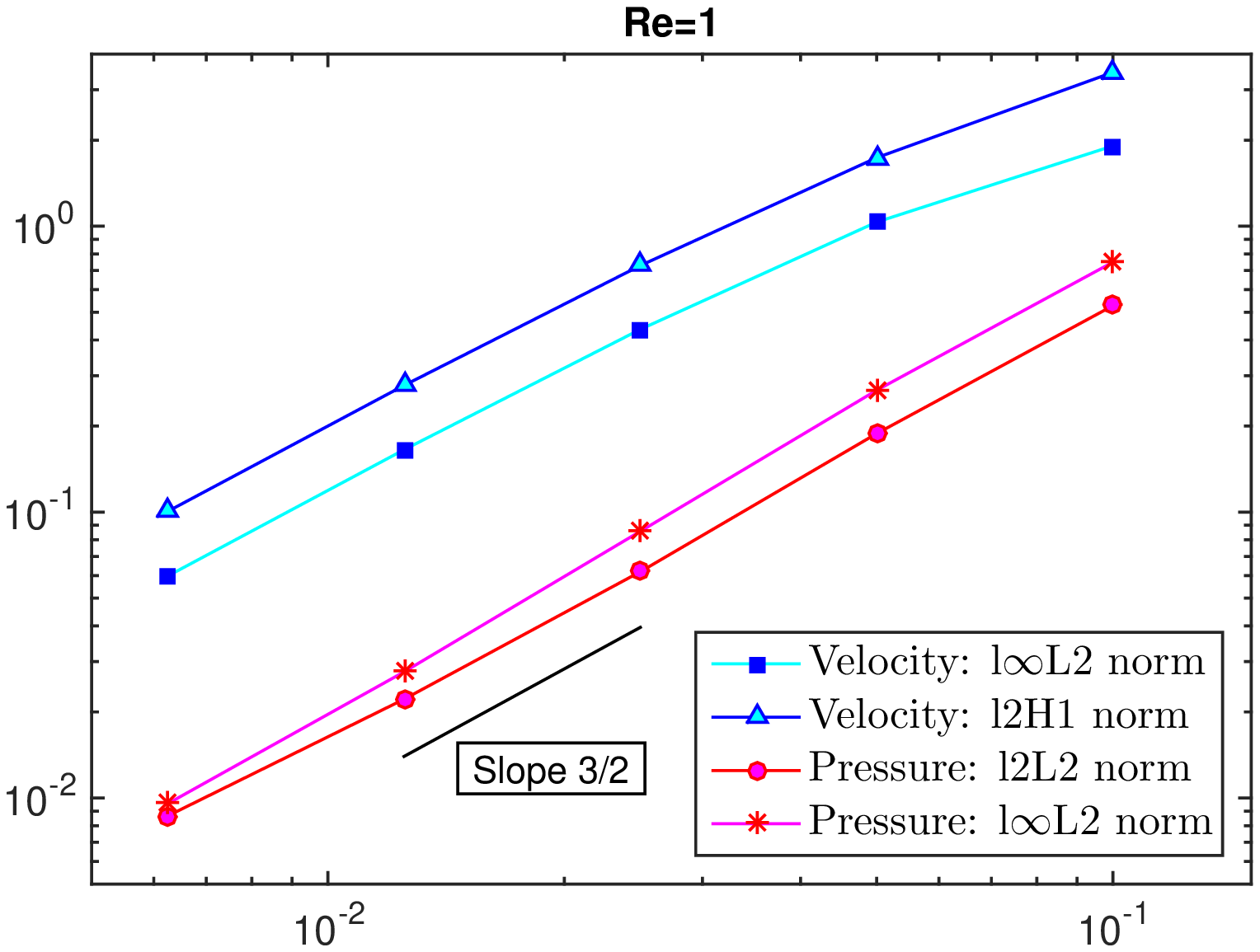}
\caption{$\Reyn=1$}
\end{subfigure} 
\begin{subfigure}[b]{0.46\textwidth}
    \includegraphics[width=\textwidth]{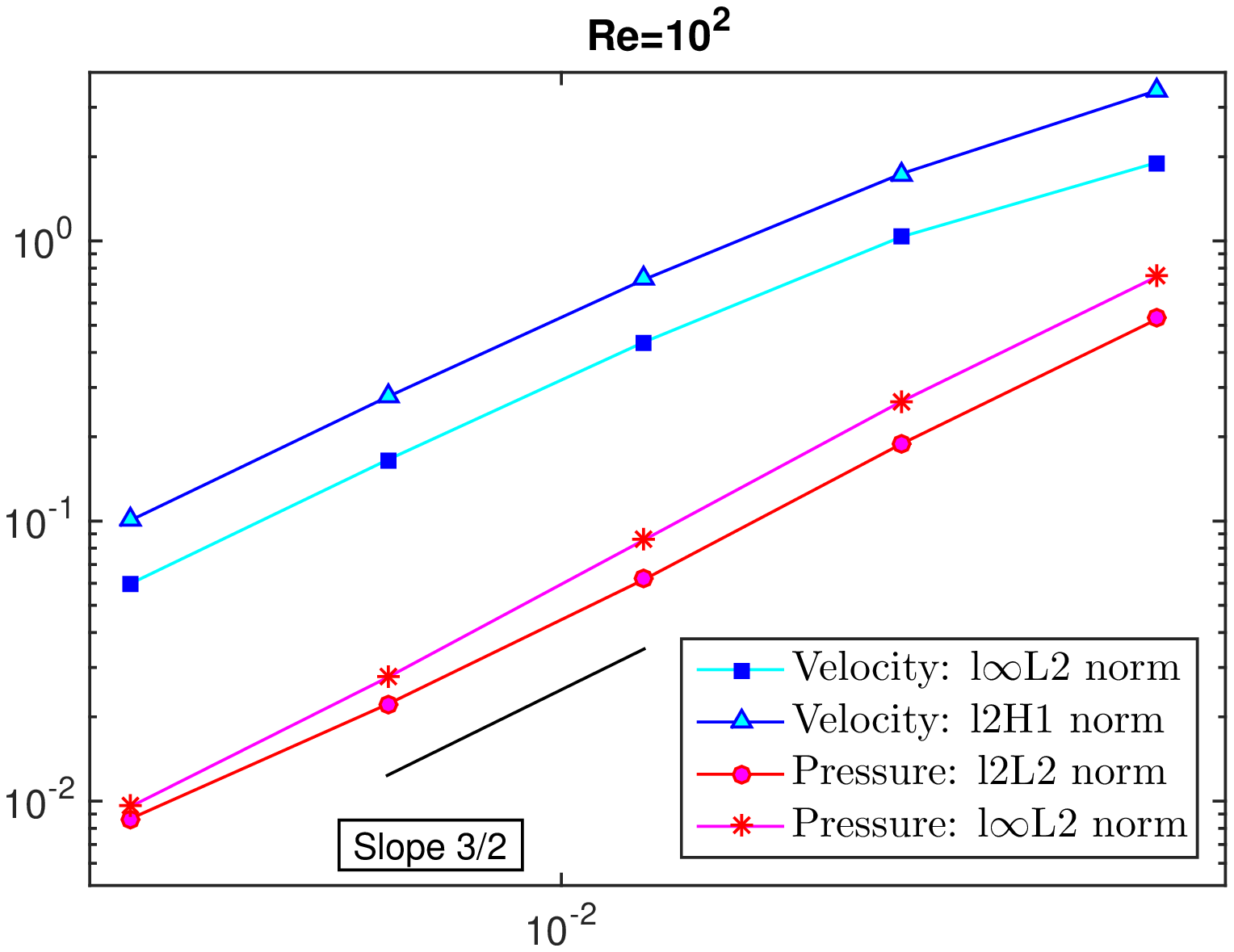}
\caption{$\Reyn=10^2$}
\end{subfigure}
\begin{subfigure}[b]{0.46\textwidth}
    \includegraphics[width=\textwidth]{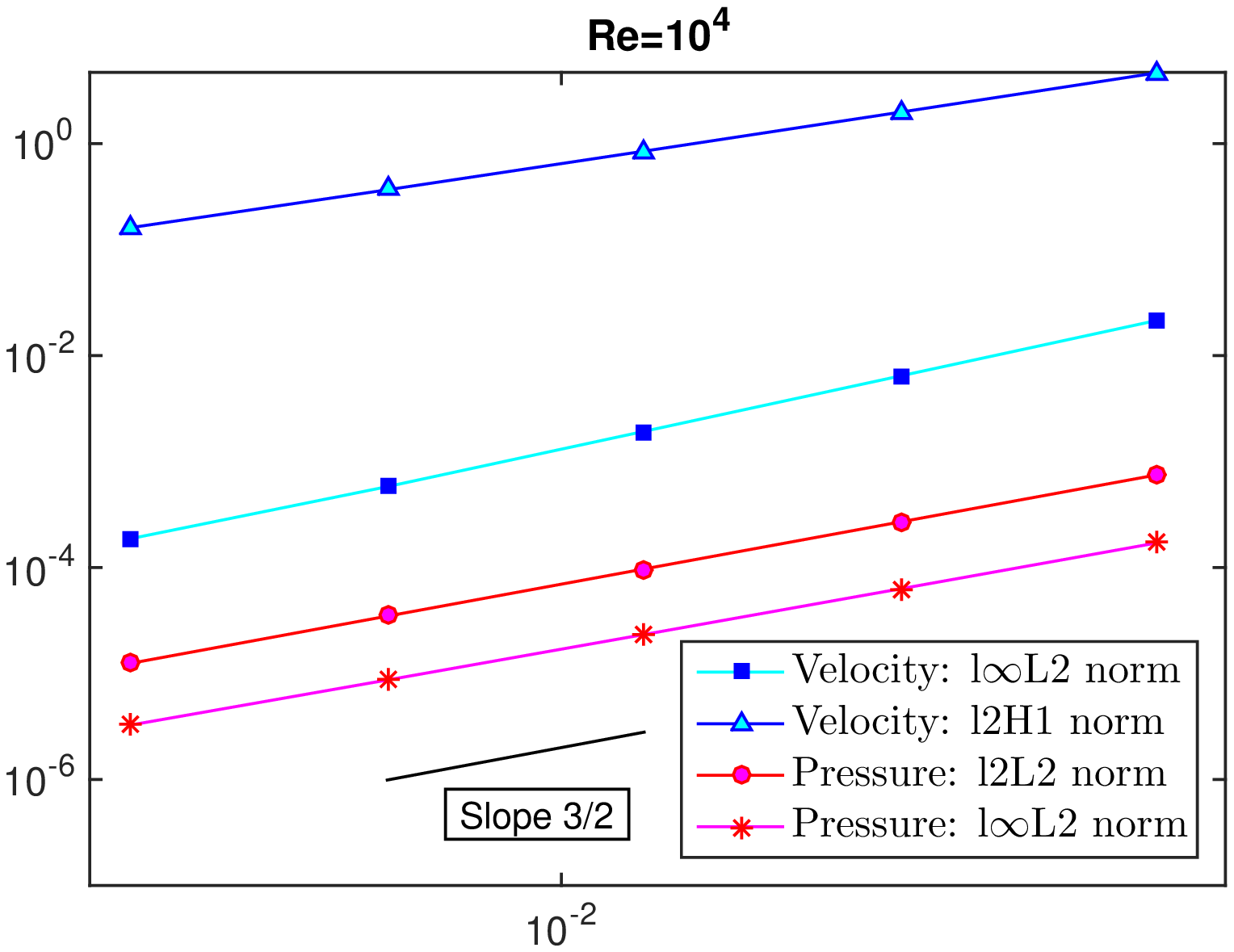}
\caption{$\Reyn=10^4$}
\end{subfigure}
\caption{Decay of different error norms versus time step $\dt$ for the rotational pressure correction projection method with boundary correction of Section~\ref{sec:Baensch} \shlee{with different $\Reyn$ numbers}. An order of convergence $\calO(\dt^{3/2})$ is observed. \shlee{Note that we employ a smaller mesh size for case a) $\Reyn=0.1$ due to the mesh condition \eqref{eq:CFL}. } } 
\label{fig:bd_correc} 
\end{figure}

To test the scheme with boundary correction \eqref{eq:velstep}--\eqref{eq:presup} of Section~\ref{sec:Baensch} we keep the setup of \S\ref{subsec:grad_div}. The results are shown in Figure~\ref{fig:bd_correc} \shlee{with different 
$\Reyn \in \{10^{-1},10^0,10^2,10^4 \}$  numbers.} 
The space discretization is chosen fine enough ($h=0.015625$) so that it does not pollute the time discretization error. 
We observe a rate of convergence of $\calO(\dt^{3/2})$ for the velocity in the $\ell^2(\bH^1(\Omega))$ and $\ell^{\infty}(\bL^2(\Omega))$ norms and the pressure in the $\ell^2(L^2(\Omega))$ and $\ell^{\infty}(L^2(\Omega))$ norms for all cases. 
From these observations we can obtain several conclusions: The first, and obvious one, is that indeed the boundary correction provides an improvement in accuracy over the standard rotational scheme.
{\color{blue} Secondly, to obtain stability and expected convergence rate, we observe that, while \eqref{eq:CFL} might not be sharp, 
the time step $\dt$ is related to the Reynolds number $\Reyn$ and the mesh size $h$ especially when $\Reyn < 1$. This can be illustrated by the fact that, in the case of $\Reyn=10^{-1}$ a smaller time step and mesh size ($h=0.0078125$) were required to obtain the expected convergence rates.}

One final observation is in convergence order. 
Namely, even though the scheme with grad-div stabilization couples the components of the velocity and thus complicates the linear algebra, a comparison of Figures \ref{fig:grad_div} (b) and \ref{fig:bd_correc} (b) with $\Reyn=1$ reveals that the magnitude of the errors for both velocity and pressure is smaller for a same given mesh size and a time step. Therefore, we observe that grad-div stabilization scheme seems to be more accurate.

Finally, we must remark that the computations of \cite{MR3253466} show a rate of convergence of $\calO(\dt^2)$. However, we believe that this is due to the fact that the author there considers a smooth domain and that, for general domains, our results are sharp.

\subsection{Flow around a cylinder}
\label{sub:cylinder}
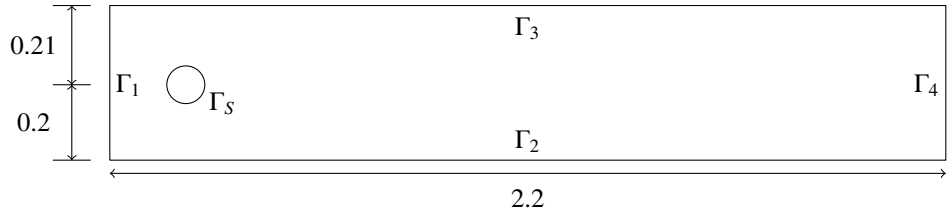
\begin{figure}[!h]
\centering
\begin{tikzpicture}[scale = 5]
  \draw (0.2,0.2) circle (0.05);
  \draw (0,0) -- (2.2,0) -- (2.2,0.41) -- (0,0.41) -- (0,0);
  
  \draw (-0.15,0) -- (-0.05,0);
  \draw (-0.15,0.2) -- (-0.05,0.2);
  \draw (-0.15,0.41) -- (-0.05,0.41);
  
   \draw[<->] (-0.1,0) -- (-0.1,0.2);
   \draw[<->] (-0.1,0.2) -- (-0.1,0.41);
   
   \node[] (a) at (-0.2,0.1) {$0.2$};
   \node[] (b) at (-0.2,0.305) {$0.21$};

   \node[] (c) at (0.05,0.2) {$\Gamma_1$};
   \node[] (c) at (0.3,0.15) {$\Gamma_S$};   

   \draw[<->] (0.,-0.035) -- (2.2,-0.035);
   \node[] (a) at (1.1,-0.1) {$2.2$};
   
   \node[] (a) at (1.1,0.05) {$\Gamma_2$};
   \node[] (a) at (1.1,0.35)  {$\Gamma_3$};  
   \node[] (a) at (2.15,0.2)  {$\Gamma_4$};   
\end{tikzpicture}
\caption{Geometry for the test case of \S\ref{sub:cylinder} with the notation for different pieces of the boundary. Homogeneous no-slip boundary conditions are imposed on $\Gamma_S \cup \Gamma_2 \cup \Gamma_3$, outflow boundary conditions are imposed on $\Gamma_4$, \ie \eqref{eq:BC4}. A parabolic inflow, as in \eqref{eq:Couette}, is prescribed on $\Gamma_1$.}
\label{fig:setup}
\end{figure} 
\begin{figure}[!h]
  \begin{center}
   \includegraphics[scale=0.25]{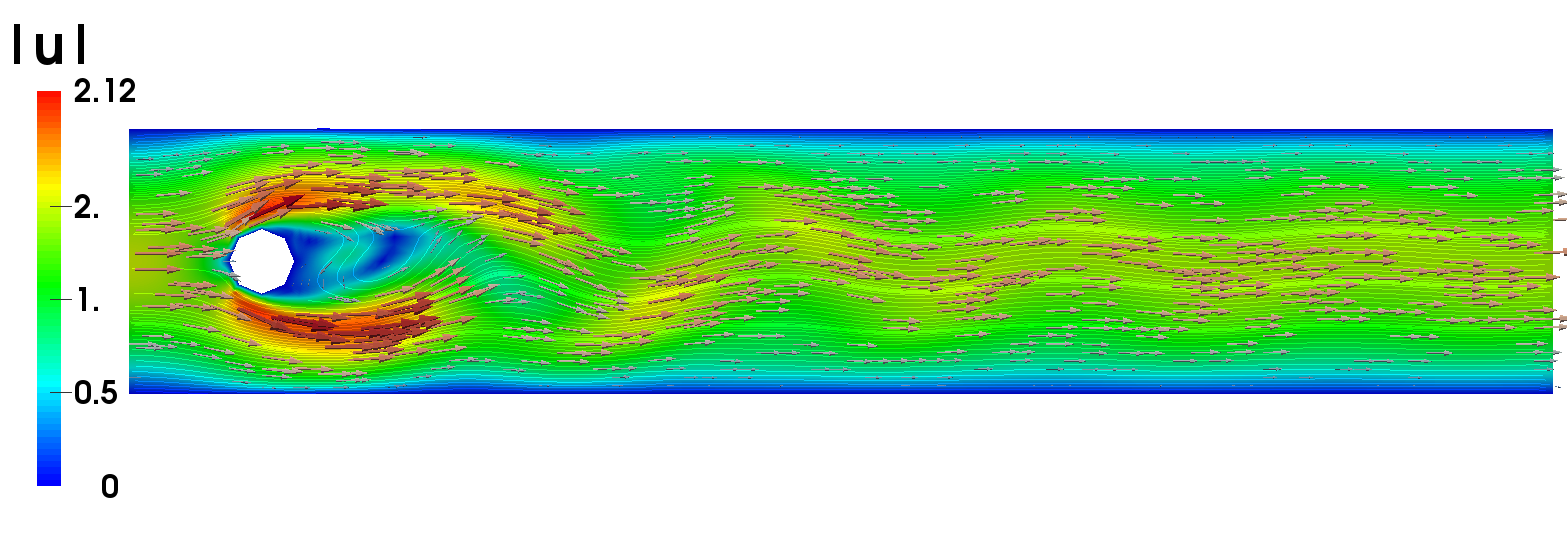}
   \includegraphics[scale=0.25]{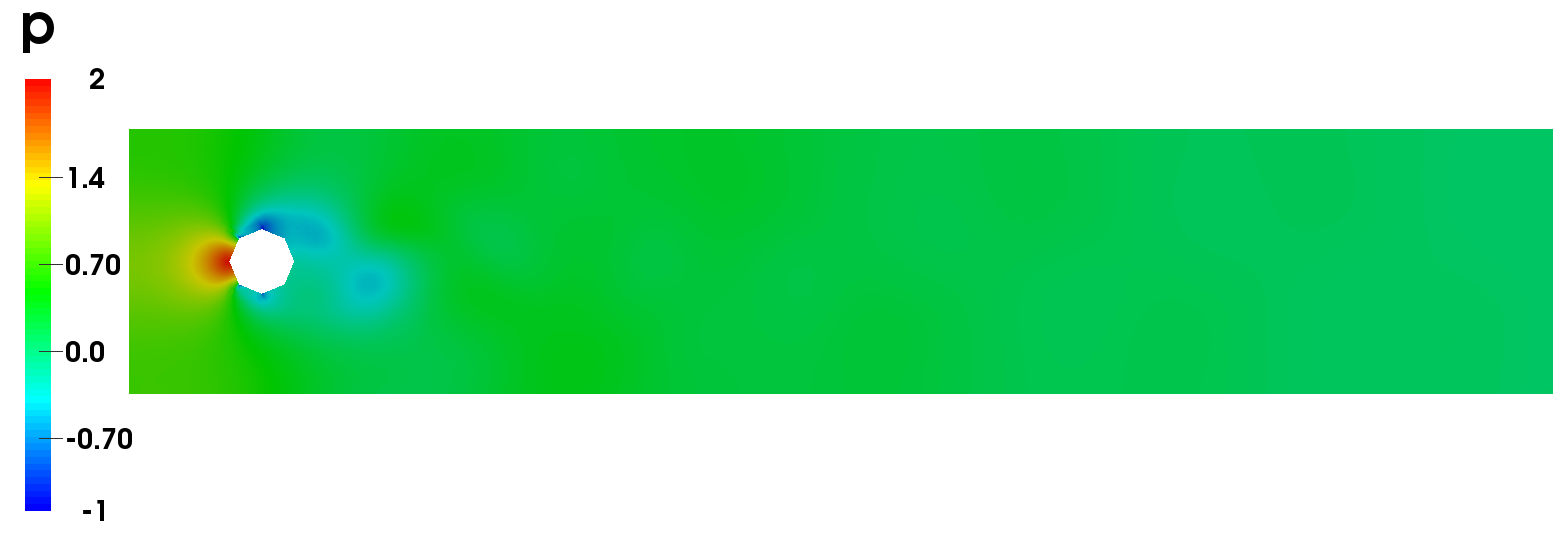}
  \end{center}
  \caption{Velocity magnitude and pressure values at $t=30$ for the example of \S\ref{sub:cylinder}.} 
\label{fig:result} 
\end{figure}

Let us conclude with a more realistic example. In the domain $\Omega = (0,2.2) \times (0,0.41) \backslash B_r(0.2.0.2)$, we compute 
a 2D laminar flow around a cylinder 
which is introduced as a benchmark problem in \cite{schafer1996benchmark}. 
Here  $B_r:= \{ \bx \in \Real^2 : \ |\bx - \bc| \leq r \}$ is the circle with the center $\bc = (0.2,0.2)^\intercal$ and radius $r=0.05$.
The detailed geometry and the notation we adopt for pieces of the boundary are depicted in Figure~\ref{fig:setup}. 
We set homogeneous no-slip boundary conditions on $\Gamma_S \cup \Gamma_2 \cup \Gamma_3$, \ie
\[
  \ue_{|\Gamma_S \cup \Gamma_2 \cup \Gamma_3} = \boldsymbol{0}.
\]
On the outflow boundary $\Gamma_4$, a homogeneous \AJS{open boundary} condition is given:
\begin{equation}
\label{eq:BC4}
  \left( \frac1\Reyn \AJS{\GRAD \ue} - \pe \polI \right)_{|\Gamma_4} \SCAL \bn = \boldsymbol{0}.
\end{equation}
Finally, on $\Gamma_1$, a parabolic inflow is prescribed
\begin{equation}
\label{eq:Couette}
  \ue(t,0,y) = \left( \frac{4y(0.41 - y)\textsf{U}}{0.41^2}, 0 \right)^\intercal,
\end{equation}
where the velocity magnitude is $\textsf{U} =1.5$. The fluid begins at rest and we set the Reynolds number  to $\Reyn = 1000$. To ensure numerical stability we add a first order numerical viscosity to the momentum equation of the form $\DIV ( \nu \GRAD \ue^{k+1}) $ where, in every cell, the artificial viscosity is defined by $\nu := \frac12 c h | \ue^k |$, with $h$ being the cell size. In all tests we set $c = 0.2$.

For space discretization we employ the lowest order Taylor-Hood elements $\polQ_2/\polQ_1$ with $33,536$ and $4,288$ degrees of freedom for velocity and pressure, respectively. The minimum mesh size in the domain is $h=0.125$ The time step is chosen as $\tau = 0.005$.

Figure~\ref{fig:result} depicts the magnitude of the velocity $| \ue |$ and the pressure values at time $t=30$ by using the boundary correction method of Section~\ref{sec:Baensch}. The grad-div stabilization method of Section~\ref{sec:Sanghyun} yields similar results, albeit at a higher computational cost due to the complexity of the underlying linear systems, see \cite{SLee2013}. We note that the obtained results seem to coincide with those shown in \cite{schafer1996benchmark}. 

\section{Conclusion}
We present stability analyses for the two different rotational pressure correction fractional time stepping schemes supplemented with open and traction boundary conditions. The grad-div stabilized scheme is unconditionally stable and our results can open the door for producing an error analysis for it. For the results of boundary correction scheme in Section~\ref{sec:Baensch}, we observe a major drawback. Namely, we require the rather stringent condition \eqref{eq:CFL}. Computations presented in \cite{MR3253466} and in this work indicate that this is not sharp. How to circumvent this is currently under investigation.

\section*{Acknowledgments}
The work of S. Lee is supported by the Center for Subsurface Modeling,  Institute for Computational Engineering and Sciences at UT Austin. The work of A.J.~Salgado is supported in part by the National Science Foundation grant DMS-1418784.

\bibliographystyle{plain} 
\bibliography{biblio}

\begin{thebibliography}{10}

\bibitem{MR2957735}
P.~Angot, J.-P. Caltagirone, and P.~Fabrie.
\newblock A fast vector penalty-projection method for incompressible
  non-homogeneous or multiphase {N}avier-{S}tokes problems.
\newblock {\em Appl. Math. Lett.}, 25(11):1681--1688, 2012.

\bibitem{MR2855967}
P.~Angot, J.-P. Caltagirone, and P.~Fabrie.
\newblock A new fast method to compute saddle-points in constrained
  optimization and applications.
\newblock {\em Appl. Math. Lett.}, 25(3):245--251, 2012.

\bibitem{Angot12}
P.~Angot and R.~Cheaytou.
\newblock Vector penalty-projection method for incompressible fluid flows with
  open boundary conditions.
\newblock In A.~Handlovi{\v c}ov{\'a}, Z.~Minarechov{\'a}, and D.~{\v S}ev{\v
  c}covi{\v c}, editors, {\em ALGORITMY 2012, 19th Conference on Scientific
  Computing, Vysok{\'e} Tatry--Podbansk{\'e}}, pages 219--229. Slovak
  University of Technology in Bratislava,, 2012.

\bibitem{dealII83}
W.~Bangerth, T.~Heister, L.~Heltai, G.~Kanschat, M.~Kronbichler, M.~Maier, and
  B.~Turcksin.
\newblock The \texttt{deal.II} library, version 8.3.
\newblock {\em preprint}, 2015.

\bibitem{MR1826211}
E.~B{\"a}nsch.
\newblock Finite element discretization of the {N}avier-{S}tokes equations with
  a free capillary surface.
\newblock {\em Numer. Math.}, 88(2):203--235, 2001.

\bibitem{MR3253466}
E.~B{\"a}nsch.
\newblock A finite element pressure correction scheme for the {N}avier-{S}tokes
  equations with traction boundary condition.
\newblock {\em Comput. Methods Appl. Mech. Engrg.}, 279:198--211, 2014.

\bibitem{SLee2013}
A.~Bonito, J.-L. Guermond, and S.~Lee.
\newblock Modified pressure-correction projection methods: Open boundary and
  variable time stepping.
\newblock In {\em Numerical Mathematics and Advanced Applications - ENUMATH
  2013}, volume 103 of {\em Lecture Notes in Computational Science and
  Engineering}, pages 623--631. Springer, 2015.

\bibitem{FLD:FLD4071}
A.~Bonito, J-L. Guermond, and S.~Lee.
\newblock Numerical simulations of bouncing jets.
\newblock {\em International Journal for Numerical Methods in Fluids},
  80(1):53--75, 2016.

\bibitem{MR1423081}
C.-H. Bruneau and P.~Fabrie.
\newblock New efficient boundary conditions for incompressible
  {N}avier-{S}tokes equations: a well-posedness result.
\newblock {\em RAIRO Mod\'el. Math. Anal. Num\'er.}, 30(7):815--840, 1996.

\bibitem{MR3327983}
S.~Dong and J.~Shen.
\newblock A pressure correction scheme for generalized form of energy-stable
  open boundary conditions for incompressible flows.
\newblock {\em J. Comput. Phys.}, 291:254--278, 2015.

\bibitem{MR3235759}
H.C. Elman, D.J. Silvester, and A.J. Wathen.
\newblock {\em Finite elements and fast iterative solvers: with applications in
  incompressible fluid dynamics}.
\newblock Numerical Mathematics and Scientific Computation. Oxford University
  Press, Oxford, second edition, 2014.

\bibitem{MR2177143}
J.-L. Guermond, P.~Minev, and J.~Shen.
\newblock Error analysis of pressure-correction schemes for the time-dependent
  {S}tokes equations with open boundary conditions.
\newblock {\em SIAM J. Numer. Anal.}, 43(1):239--258 (electronic), 2005.

\bibitem{MR2250931}
J.-L. Guermond, P.~Minev, and J.~Shen.
\newblock An overview of projection methods for incompressible flows.
\newblock {\em Comput. Methods Appl. Mech. Engrg.}, 195(44-47):6011--6045,
  2006.

\bibitem{GuermondSalgadoErrorAnalysis}
J.-L. Guermond and A.J. Salgado.
\newblock Error analysis of a fractional time-stepping technique for
  incompressible flows with variable density.
\newblock {\em SIAM Journal on Numerical Analysis}, 49(3):917--944, 2011.

\bibitem{Gue.S02_b}
J.-L. Guermond and J.~Shen.
\newblock On the error estimates for the rotational pressure-correction
  projection methods.
\newblock {\em Math. of Comp.}, 73(248):1719--1737 (electronic), 2004.

\bibitem{MR1043610}
J.G. Heywood and R.~Rannacher.
\newblock Finite-element approximation of the nonstationary {N}avier-{S}tokes
  problem. {IV}. {E}rror analysis for second-order time discretization.
\newblock {\em SIAM J. Numer. Anal.}, 27(2):353--384, 1990.

\bibitem{KIM1985308}
J~Kim and P~Moin.
\newblock Application of a fractional-step method to incompressible
  navier-stokes equations.
\newblock {\em Journal of Computational Physics}, 59(2):308 -- 323, 1985.

\bibitem{MR3337644}
S.~Lee.
\newblock {\em Numerical simulations of bouncing jets}.
\newblock ProQuest LLC, Ann Arbor, MI, 2014.
\newblock Thesis (Ph.D.)--Texas A\&M University.

\bibitem{MR2177795}
R.H. Nochetto and J.-H. Pyo.
\newblock The gauge-{U}zawa finite element method. {I}. {T}he {N}avier-{S}tokes
  equations.
\newblock {\em SIAM J. Numer. Anal.}, 43(3):1043--1068 (electronic), 2005.

\bibitem{MR2988624}
A.~Poux, S.~Glockner, E.~Ahusborde, and M.~Aza{\"{\i}}ez.
\newblock Open boundary conditions for the velocity-correction scheme of the
  {N}avier-{S}tokes equations.
\newblock {\em Comput. \& Fluids}, 70:29--43, 2012.

\bibitem{Pyo}
J.-H. Pyo.
\newblock A classification of the second order projection methods to solve the
  {N}avier-{S}tokes equations.
\newblock {\em Korean J. Math.}, 22(4):645--658, 2014.

\bibitem{schafer1996benchmark}
M.~Sch{\"a}fer and S.~Turek.
\newblock Benchmark computations of laminar flow around a cylinder.
\newblock In {\em Flow Simulation with High-Performance Computers II.}, pages
  547--566. Springer, 1996.
\newblock Notes Numer. Fluid Mech., vol. 52,.

\bibitem{TMV96}
L.J.P. Timmermans, P.D. Minev, and F.N. {van de Vosse}.
\newblock An approximate projection scheme for incompressible flow using
  spectral elements.
\newblock {\em Internat. J. Numer. Methods Fluids}, 22:673--688, 1996.

\end{thebibliography}

\end{document}